\documentclass[11pt]{amsart}

\usepackage{amssymb,latexsym, mathtools,extpfeil,amsmath,comment}
\usepackage{amsfonts}
\usepackage{multirow}
\usepackage{dsfont}
\usepackage{youngtab,ytableau}
\usepackage[pdftex]{hyperref}
\usepackage{color}
\usepackage{graphicx}
\usepackage[all]{xy}
\usepackage{tikz-cd}
\usetikzlibrary{cd}

\numberwithin{equation}{section}

\newtheorem{thm}{Theorem}[section]
\newtheorem{lem}[thm]{Lemma}
\newtheorem{cor}[thm]{Corollary}
\newtheorem{prop}[thm]{Proposition}

\theoremstyle{remark}
\newtheorem{rem}[thm]{Remark}
\newtheorem{example}[thm]{Example}
\theoremstyle{definition}

\theoremstyle{definition}

\newtheorem{question}[thm]{Question}
\theoremstyle{definition}

\newtheorem{defn}[thm]{Definition}

\numberwithin{equation}{section}

\newcommand{\Z}{\mathbb Z}

\newcommand{\C}{\mathbb C}
\newcommand{\fg}{\mathfrak{g}}

\newcommand{\fb}{\mathfrak{b}}
\newcommand{\fl}{\mathfrak{l}}

\newcommand{\ft}{\mathfrak{t}}

\newcommand{\fs}{\mathfrak{s}}

\newcommand{\B}{\mathcal{B}}
\newcommand{\ca}{\mathcal{A}}

\newcommand{\ci}{\mathcal{I}}
\newcommand{\cj}{\mathcal{J}}

\DeclareMathOperator{\pt}{pt}
\DeclareMathOperator{\diag}{diag}
\DeclareMathOperator{\Inv}{Inv}
\DeclareMathOperator{\ev}{ev}
\DeclareMathOperator{\bfx}{\mathbf{x}}

\newcommand{\RSCT}{{\tt RSCT}}

\newcommand{\GP}{{\tt Sp}}

\newcommand{\tabone}{\Upsilon}
\newcommand{\tabtwo}{\Omega}

\newcommand{\map}{\eta}

\def\presuper#1#2%
  {\mathop{}%
   \mathopen{\vphantom{#2}}^{#1}%
   \kern-\scriptspace%
   #2}

\ytableausetup{boxsize=1.1em}

\begin{document}

\title[An equivariant basis for the cohomology of Springer fibers]{An equivariant basis for the cohomology of Springer fibers}
\author{Martha Precup}
\address{Department of Mathematics and Statistics\\ Washington University in St. Louis \\ St. Louis, Missouri  63130 \\ U.S.A. }
\email{martha.precup@wustl.edu}

\author{Edward Richmond}
\address{Department of Mathematics\\ Oklahoma State University  \\ Stillwater, Oklahoma 74078 \\ U.S.A. }
\email{edward.richmond@okstate.edu}


%
\maketitle

\begin{abstract} Springer fibers are subvarieties of the flag variety that play an important role in combinatorics and geometric representation theory.  In this paper, we analyze the equivariant cohomology of Springer fibers for $GL_n(\C)$ using results of Kumar and Procesi that describe this equivariant cohomology as a quotient ring.  We define a basis for the equivariant cohomology of a Springer fiber, generalizing a monomial basis of the ordinary cohomology defined by De Concini and Procesi and studied by Garsia and Procesi.  Our construction yields a combinatorial framework with which to study the equivariant and ordinary cohomology rings of Springer fibers.  As an application, we identify an explicit collection of (equivariant) Schubert classes whose images in the (equivariant) cohomology ring of a given Springer fiber form a basis.
\end{abstract}

\section{introduction}

This paper analyzes the equivariant cohomology of Springer fibers in Lie type A.  Springer fibers are fibers of a desingularization of the nilpotent cone in $\mathfrak{gl}_n(\C)$.   Springer showed that the symmetric group acts on the cohomology of each Springer fiber, the top-dimensional cohomology is an irreducible representation, and each irreducible symmetric group representation can be obtained in this way \cite{Springer1976, Springer1978}.  As a consequence, Springer fibers frequently arise in geometric representation theory and algebraic combinatorics; see \cite{Spaltenstein1976, Fresse2009, Fresse2009-2, Fung2003, Graham-Zierau2011, PT19} for just a few examples.

There is also an algebraic approach to the Springer representation for $GL_n(\C)$, as we now explain.   Motivated by a conjecture of Kraft~\cite{Kraft1981}, De Concini and Procesi \cite{DP81} gave a presentation for the cohomology of a type A Springer fiber as the quotient of a polynomial ring.   Furthermore, this identification is $S_n$-equivariant so Springer's representation can also be constructed as the symmetric group action on the quotient of a polynomial ring.  These results were generalized to the setting of other algebraic groups by Carrell in \cite{Ca86}.

The generators of the ideal defining the presentation of the cohomology of a type A Springer fiber were further simplified by Tanisaki~\cite{Tanisaki1982}. Finally, Garsia and Procesi used the aforementioned results to study the graded character of the Springer representation in \cite{GP92}.  Their work gives a linear algebraic proof that this character is closely connected to the so-called $q$-Kostka polynomials. As part of their analysis, Garsia and Procesi study a monomial basis for the cohomology ring, originally defined by De Concini and Procesi in~\cite{DP81}, with many amenable combinatorial and inductive properties.  We refer to the collection of these monomials as the \textit{Springer monomial basis}.

Let $GL_n(\C)$ denote the algebraic group of $n\times n$ invertible matrices with Lie algebra $\mathfrak{gl}_n(\C)$ of $n\times n$ matrices.   Denote by $B$ the Borel subgroup of upper triangular matrices, and by $\fb$ its Lie algebra.  Given a nilpotent matrix $X\in \mathfrak{gl}_n(\C)$, let $\lambda$ be the partition of $n$ determined by the sizes of the Jordan blocks of $X$.  The flag variety of $GL_n(\C)$ is the quotient $\B:=GL_n(\C)/B$ and the Springer fiber corresponding to $\lambda$ is defined as the subvariety
$$\B^\lambda:=\{gB\in \B \ |\ g^{-1}Xg\in \fb\}.$$
Let $T$ denote the maximal torus of diagonal matrices in $GL_n(\C)$ and $L$ be the Levi subgroup of block diagonal matrices with block sizes determined by the partition $\lambda$.  We may assume without loss of generality that $X$ is in Jordan canonical form, and hence $X$ is regular in the Lie algebra of $L$.  Moreover, the subtorus $S:=Z_G(L)_0 \subseteq T$ acts on the Springer fiber $\B^\lambda$.  We consider the equivariant cohomology $H_S^*(\B^\lambda)$.  The goal of this manuscript is provide a combinatorial framework to study this equivariant cohomology.

There is a known presentation for $H_S^*(\B^\lambda)$ given by Kumar and Procesi \cite{KP12}, and the equivariant Tanisaki ideal has been determined by Abe and Horiguchi~\cite{AH2016}.  Our work below initiates a study of $H_S^*(\B^\lambda)$ which parallels the analysis of the ordinary cohomology by Garsia and Procesi in \cite{GP92}.   We define a collection of polynomials in  $H_S^*(\B^\lambda)$ using the combinatorics of row-strict tableaux. Since these polynomials map onto the Springer monomial basis under the natural projection map from equivariant to ordinary cohomology $H_S^*(\B^\lambda)\to H^*(\B^\lambda)$, we call them \textit{equivariant Springer monomials}.  We prove that a basis of equivariant Springer monomials exists for any Springer fiber, and provide a determinant formula (see Theorem~\ref{T:P_basis_det_formula} below) for the structure constants of any element of $H_S^*(\B^\lambda)$ with respect to this basis.

As an application, we use the algebraic and combinatorial framework developed in this manuscript to study the images of Schubert classes in $H^*(\B^\lambda)$.  Let $\phi:\B^\lambda\hookrightarrow \B$ denote the inclusion of varieties, and $\phi_0^*: H^*(\B)\to H^*(\B^\lambda)$ the induced map on ordinary cohomology.  We prove that for every partition $\lambda$, there is a natural collection of Schubert classes whose images under $\phi_0^*$ form an additive basis of $H^*(\B^\lambda)$. This result appears as Theorem~\ref{T:Schubert_Basis} in Section~\ref{sec.monomial-Schubert} below and  Corollary~\ref{cor.equiv-main-thm} contains the equivariant version of the statement.   Phrased in terms of the work of Harada and Tymoczko in \cite{HT17}, the equivariant version of Theorem~\ref{T:Schubert_Basis} says that there exists a successful game of \textit{Betti poset pinball} for each type A Springer fiber.  As a result, we can do  computations in the (equivariant) cohomology ring more easily, as combinatorial properties of (double) Schubert polynomials are well-studied (c.f., for example, \cite{Manivel}).  Bases of this kind have been used to do Schubert calculus style computations in the equivariant cohomology rings of other subvarieties of the flag variety \cite{HT11, Drellich2015}; the authors will explore analogous computations for Springer fibers in future work.

Our Theorem~\ref{T:Schubert_Basis} generalizes results of Harada--Tymoczko \cite{HT17} and Dewitt--Harada \cite{DH} which address the case of $\lambda=(n-1,1)$ and $\lambda=(n-2,2)$, respectively.   The main difficulty in generalizing the methods used in those papers is that the equivariant cohomology classes in $H_S^*(\B^\lambda)$ constructed via poset pinball may not satisfy upper triangular vanishing conditions (with respect to some partial ordering on the set of $S$-fixed points of $\B^\lambda$).  The methods used to prove Theorem~\ref{T:Schubert_Basis} side-step this difficulty by making use of the equivariant Springer monomials.  Combining our determinantal formula for the structure coefficients of this basis with known combinatorial properties of the Schubert polynomials yields the desired result.

Let $\B_{w}:=\overline{BwB/B}$ denote the Schubert variety corresponding to a permutation $w\in S_n$.  Recall that the Schubert polynomial $\mathfrak{S}_w(\mathbf{x})$ represents the fundamental cohomology class of the Schubert variety $\B_{w_0w}$ where $w_0$ denotes the longest element in $S_n$.  That is, $\mathfrak{S}_w(\mathbf{x})$ is a polynomial representative for the cohomology class $\sigma_{w}\in H^*(\B)$ defined uniquely by the property that $\sigma_w\cap [\B]=[\B_{w_0w}]$. Here $[\B]$ and $[\B_{w_0w}]$ denote the fundamental homology classes of $\B$ and $\B_{w_0w}$, respectively, and $\cap [\B]:H^*(\B)\rightarrow H_*(\B)$ denotes the Poincar\'e duality isomorphism obtained by taking the cap product with the top fundamental class.

In this paper, we study the polynomials $\phi_0^*(\mathfrak{S}_w(\mathbf{x}))$ in $H^*(\B^\lambda)$ from a combinatorial perspective. On the other hand, each is a polynomial representative for the cohomology class $\phi_0^*(\sigma_w)$ and it is natural to ask if these classes have geometric meaning. In the last section, we show that the classes $\phi_0^*(\sigma_w)$ play an analogous role with respect to the homology of $\B^\lambda$ as that played by the Schubert classes with respect to the homology of $\B$. More precisely, we prove in Proposition \ref{P:homology_classes} below that
\[
\phi_0^*(\sigma_w)\cap [\B^\lambda]=[\B^\lambda\cap g\B_{w_0w}]\in H_*(\B^\lambda)
\]
for generic $g\in GL_n(\C)$. Here $\cap [\B^\lambda]:H^*(\B^\lambda)\rightarrow H_*(\B^\lambda)$ denotes capping with the top fundamental class $[\B^\lambda]\in H_*(\B^\lambda)$. Since $\B^\lambda$ is typically not smooth, this map is not an isomorphism of groups.

The remainder of the paper is structured as follows.  The next section covers the necessary background information and notation needed in later sections, including a presentation of the equivariant cohomology of the Springer fiber due to Kumar and Procesi.  The third and fourth sections of this paper establish the combinatorial groundwork for our study of $H^*_S(\B^\lambda)$.  We use row-strict composition tableaux to define an equivariant generalization of the Springer monomial basis in Section~\ref{sec.tableaux}, called the equivariant Springer monomials, and develop the structural properties of these polynomials further in Section~\ref{sec.determinant}.  In particular,  we give a determinant formula for the structure coefficients of $H_S^*(\B^\lambda)$ with respect to the basis of equivariant Springer monomials in Theorem~\ref{T:P_basis_det_formula} of Section~\ref{sec.determinant}.   Finally, Section~\ref{sec.monomial-Schubert} uses the equivariant Springer monomials to study the images of monomials and Schubert polynomials in the cohomology of Springer fibers.  Our main result in Section~\ref{sec.monomial-Schubert} is Theorem~\ref{T:Schubert_Basis}, which was discussed above.   We conclude with an analysis of the geometric meaning of the classes $\phi_0^*(\sigma_w)$ in Section~\ref{sec.geometry}.

 \textit{Acknowledgements.} The authors are grateful to Alex Woo, Jim Carrell, Dave Anderson, Anand Patel, Prakash Belkale, Jeff Mermin, and Vasu Tewari for helpful conversations and feedback.  The first author was supported by a Oklahoma State University CAS summer research grant.  The second author was partially supported by an AWM-NSF travel grant and NSF grant DMS 1954001 during the course of this research.


\section{Background}\label{sec.set-up}

As in the introduction, let $G=GL_n(\C)$ and $\fg=\mathfrak{gl}_n(\C)$ denote its Lie algebra.  Denote by $T$ the maximal torus of diagonal matrices in $G$ and by $B$ the Borel subgroup of upper triangular matrices.  Let $\fb$ denote the Lie algebra of $B$. The Weyl group of $G$ is $W\simeq S_n$.  We let $s_i$ denote the simple transposition exchanging $i$ and $i+1$.  Throughout this manuscript, $\alpha=(\alpha_1,\ldots \alpha_k)$ denotes a (strong) composition of $n$. We call the partition of $n$ obtained by sorting the parts of $\alpha$ into weakly decreasing order the underlying partition shape of $\alpha$.

The composition ${\alpha}$ uniquely determines a standard Levi subgroup $L$ in $G$, namely the subgroup of block diagonal matrices such that the $i$-th diagonal block has dimension ${\alpha_i}\times {\alpha}_i$.  We denote the Weyl group for $L$ by $W_L$.  Let $X_\alpha:\C^n \to\C^n$ be a principal nilpotent element of $\fl$, the Lie algebra of $L$. Note that by construction, $X_\alpha\in \fg$ is a nilpotent matrix of Jordan type $\lambda$, where $\lambda$ is the underlying partition shape of $\alpha$.

Let $\B:=G/B$ denote the flag variety. The  \textbf{Springer fiber} of $X_\alpha$ is defined to be
$$\B^\alpha:=\{gB\in \B \mid g^{-1}X_\alpha g\in  \fb\}.$$
If two compositions have the same underlying partition shape, then the corresponding Springer fibers are isomorphic.  However, taking different compositions corresponding to the same partition shape yields actions of different sub-tori of $T$ on the corresponding Springer fibers.  This ultimately leads to the construction of different bases for the equivariant cohomology ring of $\B^\alpha$.

Let $S$ denote the connected component of the centralizer of $L$ in $G$ containing the identity, so $S\subseteq T$. Since $X_\alpha\in \fl$, we get that $S$ centralizes $X_\alpha$ and therefore $S$ acts on $\B^\alpha$ by left multiplication.   The purpose of this manuscript is to study the equivariant cohomology ring $H_S^*(\B^\alpha)$. We begin by reviewing a presentation for $H_S^*(\B^\alpha)$ due to Kumar and Procesi~\cite{KP12}.

\subsection{A presentation of $H_S^*(\B^\alpha)$}\label{sec.equiv-presentation}

Recall from the introduction that $\phi:\B^\alpha\hookrightarrow \B$ denotes the inclusion map of $\B^\alpha$ into the flag variety and consider the induced map on equivariant cohomology, $\phi^*: H_T^*(\B)\rightarrow H_S^*(\B^\alpha)$.  In this paper, we work with singular and equivariant cohomology with coefficients in $\C$.
Note that $\phi^*$ naturally factors through $H_S^*(\B)$,
\[
\phi^*: H_T^*(\B)\rightarrow H_S^*(\B)\rightarrow H_S^*(\B^\alpha).
\]
Let $\ft$ denote the Lie algebra of $T$.  The coordinate ring of $\ft\times\ft$ is the polynomial ring
\begin{eqnarray}\label{e.polynomial}
\C[\ft \times \ft]\simeq S(\ft^*)\otimes S(\ft^*)\simeq \C[y_1,\ldots,y_n;\, x_1,\ldots x_n]/\mathcal{J}
\end{eqnarray}
where $\mathcal{J}$ is the ideal $\left< \sum_i y_i, \sum_i x_i\right>$.

It is well known that the $T$-action on $\B$ by left multiplication is equivariantly formal, implying
$$H_T^*(\B)\simeq H_T^*(\pt)\otimes_{\C} H^*(\B).$$
Since $H_T^*(\pt)\simeq S(\ft^*)$ we have that $H_T^*(\B)$ is a free $S(\ft^*)$-module. Recall that the {Borel homomorphism},
$$\beta:S(\ft^*)\rightarrow H_T(\B)$$
is defined by $\beta(x_i)=-c_1(\mathcal{L}_i)$, where $c_1(\mathcal{L}_i)$ is the $T$-equivariant first Chern class of  $\mathcal{L}_i$, the $i$-th line bundle of the tautological filtration of sub-bundles on $\B$.  In other words, the fiber of $\mathcal{L}_i$ over a flag $V_\bullet\in \B$ is the line $V_i/V_{i-1}$.  This map induces a surjective algebra homomorphism,
\begin{eqnarray*}\label{e.flagvtymap}
\chi:\C[\ft\times\ft]\twoheadrightarrow H_T^*(\B)
\end{eqnarray*}
given by $\chi(p\otimes q)=p\cdot\beta(q)$ where $p\in S(\ft^*)$.  Following the work the Kumar in Procesi~\cite{KP12}, we define $\theta:\C[\ft \times \ft]\rightarrow H_S^*(\B^\alpha)$ to be the composition of maps
\begin{eqnarray}\label{eqn.theta}
\theta: \C[\ft\times\ft]\xtwoheadrightarrow{\chi} H_T^*(\B)\xtwoheadrightarrow{\phi^*} H_S^*(\B^\alpha).
\end{eqnarray}

Let $\fs\subseteq \ft$ denote the Lie algebra of $S$ and $Z_\alpha$ be the reduced closed subvariety of $\ft\times\ft$ defined by
\[
Z_\alpha:=\{(h,wh)\mid h\in \fs,\, w\in W\}.
\]
Note that we may also view $Z_\alpha$ as a subvariety of $\fs\times \ft\subseteq \ft\times\ft$, and we use this perspective in our computations below.  Let $\ci(Z_\alpha)\subseteq \C[\ft\times \ft]$ denote the vanishing ideal of $Z_\alpha$.  The coordinate ring
\begin{eqnarray}\label{e.coordring}
\ca:=\C[Z_\alpha]\simeq \C[\ft\times\ft]/\ci(Z_\alpha)
\end{eqnarray}
is naturally an $S(\fs^*)$-algebra via the projection $Z_\alpha\rightarrow \fs$ onto the first factor.  Moreover, the ring $\ca$ inherits a non-negatively graded structure from $\C[\ft\times\ft]$.  We also define the graded $\C$-algebra
\begin{eqnarray}\label{e.coordring_0}
\ca_0:=\C\otimes_{S(\fs^*)}\ca
\end{eqnarray}
where $\C$ is considered an $S(\fs^*)$-module under evaluation at 0.  Note that if $\alpha=(1,1,\ldots,1)$, then $\B^\alpha=\B$ and $S=T$.  In this special case, we denote the corresponding coordinate ring by $\ca':=\C[Z_{(1,1,\ldots,1)}]$.  The next theorem from~\cite{KP12} gives a presentation of $H_S^*(\B^\alpha)$.

\begin{thm}[Kumar--Procesi]\label{T:KP_theorem}
The kernel of the map $\theta:\C[\ft\times\ft]\rightarrow H_S^*(\B^\alpha)$ defined in~\eqref{eqn.theta} is the ideal $\ci(Z_\alpha)$.  In particular, $\theta$ induces a graded $S(\fs^*)$-algebra isomorphism
$$\bar{\theta}: \ca \rightarrow H_S^*(\B^\alpha),$$
making the following diagram commute.
\[\begin{tikzcd}
\ca' \arrow[r, "\bar\theta'"] \arrow[d] & H_T^*(\B) \arrow[d, "\phi^*"] \\
\ca \arrow[r, "\bar\theta"]
& H_S^*(\B^\alpha)
\end{tikzcd}\]
Furthermore, the map $\bar{\theta}$ naturally descends to a $\C$-algebra isomorphism:
$$ \bar{\theta}_0:\ca_0 \rightarrow H^*(\B^\alpha)$$
with the following commutative diagram.
\[\begin{tikzcd}
\ca_0' \arrow[r, "\bar\theta'_0"] \arrow[d] & H^*(\B) \arrow[d, "\phi^*_0"] \\
\ca_0 \arrow[r, "\bar\theta_0"]
& H^*(\B^\alpha)
\end{tikzcd}\]

\end{thm}

Since the map $\bar\theta$ is an isomorphism, we will use $\phi^*,\phi^*_0$ for the respective restriction maps $\phi^*:\ca'\rightarrow\ca$  and $\phi^*_0:\ca'_0\rightarrow\ca_0$.  In particular, if $\mathfrak{S}_w(\mathbf{x})$ denotes a Schubert polynomial representing the class $\sigma_w$, then the polynomial $\phi^*_0(\mathfrak{S}_w(\mathbf{x}))\in \ca_0$ represents the class $\phi^*_0(\sigma_w)\in H^*(\B^\alpha)$.

\begin{rem}\label{rem.free} It is well-known that the cohomology $H^*(\B^\alpha)$ is concentrated in even degrees \cite{Spaltenstein1976}.  Thus  the equivariant cohomology $H_S^*(\B^\alpha)$ is a free $S(\fs^*)$-module, and isomorphic to the tensor product,
\[
H_S^*(\B^\alpha) \simeq H_S^*(\mathrm{pt}) \otimes_\C H^*(\B^\alpha) \simeq S^*(\fs) \otimes_\C H^*(\B^\alpha).
\]
The graded $S(\fs^*)$-algebra isomorphism of Theorem~\ref{T:KP_theorem} implies $\ca$ is a free $S(\fs^*)$-module with rank equal to the number of $S$-fixed points of $\B^\alpha$, namely $|W/W_L|$ (c.f.~\cite[Lemma 2.1]{KP12}).
\end{rem}


\subsection{Maps of polynomial rings} \label{sec.polyset-up}

Recall that $L$ is the standard Levi subgroup associated to $\alpha$ as above and $S = Z_G(L)_0$ is the connected component of the centralizer of $L$ in $G$ containing the identity.
Since $L$ is a standard Levi subgroup, if $\ft=\diag(t_1,\ldots,t_n)\subseteq \fg$ and $\fs$ has coordinates $(z_1,\ldots,z_k)$, then the embedding of the subalgebra $\fs$ into $\ft$ is given by
\[
i:\fs\hookrightarrow \ft,\; (z_1,\ldots,z_k) \mapsto \diag((z_1)^{\alpha_1},\ldots, (z_k)^{\alpha_k})
\]
where
\[
(z_i)^{\alpha_i}:=\underbrace{z_i,\ldots, z_i}_{\alpha_i \text{ times}}.
\]
This embedding induces a map $i^*:\C[\ft\times\ft]\rightarrow\C[\fs\times\ft].$  If $F\in\C[\ft\times\ft]$, then $F$ and $i^*(F)$ have the same values on $Z_\alpha$.  This implies
\begin{eqnarray}\label{eqn.coordinatechange}
\ca:=\C[Z_{\alpha}]\simeq\C[\fs\times\ft]/\ci_{\fs}(Z_{\alpha})\simeq \C[\ft\times\ft]/\ci(Z_{\alpha})
\end{eqnarray}
where $\ci_{\fs}(Z_{\alpha})=i^*(\ci(Z_\alpha))$ denotes the vanishing ideal of $Z_{\alpha}$ as a subvariety of $\fs\times\ft$.  We let $\pi:\C[\fs\times \ft] \to \ca$ denote the canonical projection map.  By a slight abuse of notation, we will also denote the quotient $\C[\ft\times\ft]\to \ca$ by $\pi$; the isomorphism of~\eqref{eqn.coordinatechange} tells us that we may do so without loss of generality.

As in~\eqref{e.polynomial}, there are isomorphisms
$$
\C[\fs\times\ft]\simeq S(\fs^*)\otimes S(\ft^*)\simeq \C[z_1,\ldots,z_k;\, x_1,\ldots x_n]/\cj'
$$
where $\cj'$ is the ideal $\left< \sum_i \alpha_iz_i, \sum_i x_i \right>$.  Note that
$$\C[z_1,\ldots,z_k;\, x_1,\ldots x_n]/\cj'\simeq \C[z_1, \ldots, z_{k-1}; x_1, \ldots, x_{n-1}]$$
and we make this identification below whenever it is convenient (and similarly for $\C[\ft\times \ft]$).

The ring $\ca$ inherits the graded structure of $\C[\fs\times\ft]$. In particular, the degree $k$ component of $\C[\fs\times \ft]$ is  $\bigoplus_{i+j=k} S^i(\fs^*)\otimes S^j(\ft^*)$ and we denote its image under the canonical projection map $\pi: \C[\fs\times \ft]\to \ca$ by $\ca^k$.  Let
$$\ca^k_+:=\pi\left(\bigoplus_{i+j=k,\, i>0} S^i(\fs^*)\otimes S^j(\ft^*)\right)\qquad\text{and}\qquad \ca^k_0:=\pi\left(\C\otimes S^k(\ft^*)\right)$$
denote the positive degree and degree zero components of $\ca^k$ with respect to the grading of $S(\fs^*)$.  It is easy to see that $\ca^k =\ca^k_0\oplus \ca^k_+$ and $\ca_0=\bigoplus_{j\geq 0} \ca^j_0.$  There is a surjective map
\begin{eqnarray*}
\ev:S(\fs^*)\otimes S(\ft^*) \rightarrow S(\ft^*)
\end{eqnarray*}
given by evaluation at $0$.  More explicitly, if $F=p\otimes q$ with $p\in S(\fs^*)$ and $q\in S(\ft^*)$, then $\ev(F):=p(0)\cdot q$ (then extend linearly to all of $S(\fs^*)\otimes S(\ft^*)$).  This map induces an evaluation map $\ev:\ca\rightarrow\ca_0$ giving the commutative diagram:
\begin{eqnarray}\label{eqn.cd}
\begin{tikzcd}
\C[\fs\times\ft] \arrow[r, "\pi"] \arrow[d, "\ev"] & \ca \arrow[d, "\ev"] \\
\C[\ft] \arrow[r, "\pi_0"]
& \ca_0
\end{tikzcd}
\end{eqnarray}
where, by Theorem~\ref{T:KP_theorem}, $\ca \simeq H_S^*(\B^\alpha)$ and $\ca_0 \simeq H^*(\B^\alpha)$.  Note that, under these identifications, the evaluation map is simply the usual restriction map from equivariant to ordinary cohomology.  The following lemma relates bases of the ordinary and equivariant cohomology rings of $\B^\alpha$.




\begin{lem}\label{P:equiv_basis}
Suppose $b_1,\ldots,b_m$ is a homogeneous basis of $\ca_0$ and let $B_1,\ldots, B_m$ be a set of homogeneous polynomials in $\ca$ such that $\ev(B_i)=b_i$.  Then $\{B_1,\ldots,B_m\}$ is an $S(\fs^*)$-module basis of $\ca$.
\end{lem}

\begin{proof} To begin, we prove that the $S(\fs^*)$-span of $B_1, \ldots, B_m$ is $\ca$. First note that if $\deg(b_i)=\deg(B_i)=k$, then $b_i\in \ca_0^k$ and $B_i=b_i+G_i$ for some $G_i\in \ca^k_+$.   Let $F\in \ca$.
It suffices to assume that $F\in \ca^k$ for some $k$.  We proceed by induction on $k$.
Since $b_1,\ldots,b_m$ is a basis of $\ca_0=\ev(\ca)$, we can write
$$\ev(F)=\sum_{i=1}^m c_i\, b_i$$
for some $c_1, \ldots, c_m\in \C$ and hence $F=\sum_{i=1}^m c_i\, b_i +G$ for some $G\in \ca^k_+$ and $\deg(b_i)=k$ for all $c_i\neq 0$.  We now have
\begin{align*}
F&=\sum_{i=1}^m c_i\, b_i +G\\
&=\left(\sum_{i=1}^m c_i\, b_i +\sum_{i=1}^m c_i\, G_i\right) +\left(G-\sum_{i=1}^m c_i\, G_i\right)\\
&=\left(\sum_{i=1}^m c_i\, B_i \right) +\left(G-\sum_{i=1}^m c_i\, G_i\right)\\
\end{align*}
Observe that the second term of the above sum belongs to $\ca^k_+$ and is therefore of the form $\sum p'\otimes q'$ where each $q'\in S^j(\ft^*)$ from some $j<k$.  By induction, each $q'$ is a $S(\fs^*)$-linear combination of $B_1,\ldots,B_m$ and hence so is $F$.

We next prove that $\{B_1, \ldots, B_m\}$ is $S(\fs^*)$-linearly independent.  As noted in Remark~\ref{rem.free}, $\ca$ is a free $S(\fs^*)$-module of rank $m=\dim H^*(\B^\alpha)$.  Let $Q(\fs^*)$ denote the field of fractions of $S(\fs^*)\simeq \C[z_1, \ldots, z_{k-1}]$.  Since $\ca$ is a free module, the extension of scalars $Q(\fs^*)\otimes_{S(\fs^*)} \ca$ is a free module of the same rank \cite[\S10.4, Cor.~18]{DF}.  Furthermore, the polynomials $B_1, \ldots, B_m$ must also span $Q(\fs^*)\otimes_{S(\fs^*)} \ca$.   Since the extension of scalars is an $m$-dimensional vector space, $\{B_1, \ldots, B_m\}$ are $Q(\fs^*)$-linearly independent.  Any non-trivial linear relation among $B_1, \ldots, B_m$ with $S(\fs^*)$-coefficients would also be a non-trivial linear relation over $Q(\fs^*)$, contradicting the previous sentence. We conclude $\{B_1, \ldots, B_m\}$ is $S(\fs^*)$-linearly independent, as desired.
\end{proof}


\subsection{The Springer monomial basis}

We now recall the monomial basis of $\ca_0 \simeq H^*(\B^\alpha)$ defined by De Concini and Procesi in~\cite{DP81} and further analyzed by Garsia and Procesi in~\cite{GP92}.

Let $\lambda$ be a partition of $n$ with $k$ parts  and $\lambda[i]$ be the partition of $n-1$ obtained from $\lambda$ by decreasing the $i$-th part by $1$ and sorting the resulting composition so that the parts are in weakly decreasing order.  Define $\GP_\lambda' \subset \C[\ft]$ to be the collection of monomials constructed recursively as in~\cite[\S1]{GP92} by
\begin{eqnarray}\label{eqn.GPbasis}
\GP_\lambda' = \bigsqcup_{1\leq i \leq k} x_n^{i-1} \GP_{\lambda[i]}',
\end{eqnarray}
with initial condition $\GP_\lambda'=\{1\}$ for $\lambda=(1)$.  Here $x_n^{i-1}\GP_{\lambda[i]}'$ denotes the set of monomials obtained by multiplying each monomial in $\GP_{\lambda[i]}'$ by $x_n^{i-1}$.  Observe that as defined in \cite{GP92}, the monomials in $\GP_\lambda'$ are in the variables $x_2,\ldots,x_n$.  We define
$$\GP_\lambda:=w_0\GP_\lambda'$$ where the action of the longest permutation $w_0\in W$ on variables is given by $w_0\cdot x_i:=x_{n-i+1}$.  Hence the monomials in $\GP_\lambda$ are in the variables $x_1,\ldots,x_{n-1}$.  Since the ideal $\ci(Z_\alpha)$ is invariant under the action of $W$, it follows by results of De Concini and Procesi that, as graded vector spaces,
\[
\ca_0\simeq \bigoplus_{\bfx^\delta\in\GP_\lambda} \C\cdot \bfx^\delta.
\]
Here we use standard monomial notation $\bfx^{\delta}:=x_1^{\delta_1}x_2^{\delta_2}\cdots x_{n-1}^{\delta_{n-1}}\in \C[\ft]\simeq\C[x_1, \ldots, x_{n-1}]$. We refer to the basis $\GP_\lambda$ of $H^*(\B^\alpha)$ as the \textbf{Springer monomial basis}, and to its elements as \textbf{Springer monomials}. We adopt the convention throughout this manuscript that if $\mathbf{x}^\delta\in \GP_\lambda$, then we denote both $\mathbf{x}^\delta\in \C[\ft]$ and its image under the canonical projection $\pi_0: \C[\ft]\to \ca_0$ by the same symbol.
\begin{example}
Let $n=4$ and $\lambda=(2,2)$, then $\GP_{(2,2)}=\{1, x_3, x_2, x_2x_3, x_1, x_1x_3, x_1x_2\}$
\end{example}
See~\cite[\S1]{GP92} for a more detailed example.

\begin{rem}
The Springer monomials have been generalized to study the cohomology rings of other subvarieties of the flag variety.  In particular, Mbirika in \cite{Mb10} constructs an analogous set of monomials for nilpotent Hessenberg varieties (which include Springer fibers).  In a later paper, Mbirika and Tymoczko give an analogue of the Tanisaki ideal in the Hessenberg setting \cite{MT13}.
\end{rem}


\section{Row-strict Tableaux}\label{sec.tableaux}
In this section we develop a combinatorial framework to study the ring $\ca$ defined in~\eqref{e.coordring} using row-strict composition tableaux.

\subsection{Row strict composition tableaux}
For any integers $p\leq q$, we let $[p,q]$ denote the interval $[p,q]:=\{p,p+1,\ldots, q\}$.  If $p=1$, then we set $[q]:=[1,q]$.  Given any $m\leq n$, consider $\beta=(\beta_1,\ldots,\beta_k)$ a weak composition of $m$.  The composition diagram of $\beta$ is an array of boxes with $\beta_i$ boxes in the $i$-th row ordered from top to bottom (English notation).

A \textbf{shifted row-strict composition tableau} of shape $\beta$ is a labeling $\tabone$ of the composition diagram with the $m$ integers $[n-m+1,n]$ such that the values decrease from left to right in each row.  For simplicity of notation, let $\bar m:=n-m+1$.  Let $\RSCT_n(\beta)$ denote the collection of all shifted-standard row-strict tableaux of composition shape $\beta$ with content $[\bar m,n]$. Observe that if $\bar m=1$ (i.e. $\beta$ is a composition of $n$), then the content of $\beta$ is the full standard content $[n]$; in this case, we say that $\beta$ is a \textbf{row-strict composition tableau}.

\begin{example} Consider the composition $\beta=(1,2,0,1)$ with $n=5$.  In this case $m=4$ and $\bar{m} = 2$.  There are 12 row-strict composition tableaux in $\RSCT_4(\beta)$.  Indeed, note that there are $24=4!$ possible fillings of $\beta$ using the content $[2,5]$.  Furthermore, if we define two fillings to be equivalent up to the entries in each row, e.g.
\[\ytableausetup{centertableaux}
\begin{ytableau} 3\\ 4 & 2\\ \none \\ 5  \end{ytableau} \quad \sim  \quad \begin{ytableau} 3\\ 2 & 4\\ \none \\ 5  \end{ytableau}
\]
then there are precisely two tableaux in each equivalence class and each class contains a unique row-strict composition tableau.
\end{example}

Given a composition $\beta$, let $\alpha=(\alpha_1,\ldots, \alpha_k)$ be the strong composition obtained from $\beta$ by deleting any part equal to zero.  By similar reasoning as in the example above, we have that
$$|\RSCT_n(\beta)|= \frac{n!}{\alpha_1!\cdots \alpha_k!}$$
which is precisely the number of $S$-fixed points in the Springer fiber $\B^\alpha$.  Notice that if $\bar{m}>1$, then each shifted row-strict composition tableau can be associated to a unique row-strict tableau in $\RSCT_{m}(\beta)$ by the relabeling map $i\mapsto i-\bar{m}+1$.  We use the ``shifted'' terminology since it simplifies the arguments below.  Similarly, although we typically begin with a strong composition of $n$, our inductive procedures require the generality of weak compositions.

We now define a map
\begin{eqnarray}\label{eqn.induction}
\map:\RSCT_n(\beta)\rightarrow \bigsqcup_{\beta'}\RSCT_{n}(\beta')
\end{eqnarray}
where the union on the RHS is taken over all compositions $\beta'$ obtained from $\beta$ by deleting one box from any nonzero row.  Let $\map(\tabone)$ be the composition tableau obtained by removing the box from $\tabone$ which contains its smallest entry, namely $\bar m$.  For example:
\[\ytableausetup{centertableaux}
\begin{ytableau} 7 & 4\\ 6 & 1  \\5 & 3 & 2\\ \end{ytableau}\quad \xmapsto[\hspace*{1cm}]{\map}\quad  \begin{ytableau} 7 & 4\\ 6   \\5 & 3 & 2\\ \end{ytableau}
\]
In this case, the disjoint union in~\eqref{eqn.induction} is taken over $\beta'\in \{(1,2,3), (2,1,3), (2,2,2)\}$.  The map $\map$ plays an important role in the inductive arguments below; note that $\map$ is in fact a bijection.

\begin{defn} \label{defn.inv} Let $\tabone\in\RSCT_n(\beta)$.  We say that $(i,j)$ is an \textbf{Springer inversion} of $\tabone$ if there exists $j'$ in row $j$ such that $i<j'$ and either:
\begin{enumerate}
\item $j'$ appears above $i$ and in the same column, or
\item $j'$ appears in a column strictly to the right of the column containing $i$.
\end{enumerate}
Denote the set of Springer inversions of $\tabone$ by $\Inv(\tabone)$.
\end{defn}

\begin{example}\label{Ex:inversions}
Let $n=9$ and $\beta=(2,0,3,2,1)$. Consider $\tabone\in\RSCT_9(\beta)$ with content $[2,9]$:
\[\begin{ytableau}
6 & 3\\
\none\\
8 & 7 & 4 \\
5 & 2\\
9
\end{ytableau}\]
The inversions of $\tabone$ are $\Inv(\tabone)=\{(2,1), (2,3), (5,1), (5,3),(6,3)\}$.
\end{example}

\begin{rem}\label{rem.conventions} Note that the definition above is closely related to the notion of a \textit{Springer dimension pair} considered by the first author and Tymoczko in~\cite{PT19}.  In that paper, the convention is that the row-strict tableaux have \textit{increasing} entries (from left to right), while our convention is that the entries are \textit{decreasing} (from left to right).  This change in conventions is routine; to convert from one to the other, apply the permutation $w_0$ such that $w_0(i)= n-i+1$ for all $i$.  A Springer inversion from this paper corresponds to a unique Springer dimension pair as defined in~\cite{PT19} (up to transformation under $w_0$).  If $(i,j)$ is a Springer inversion then $(n-i+1,n-j'+1)$ is a Springer dimension pair, where $j'$ denotes the smallest element in row $j$ such that $i<j'$.
\end{rem}

The following lemma is a simple, but important fact about inversions.

\begin{lem}\label{lemma.inv}
Let $\tabone,\tabtwo\in \RSCT_n(\beta)$.  Let $j_\tabone,j_\tabtwo$ denote the indices of the rows containing $\bar m$ in $\tabone$ and $\tabtwo$, respectively.  Then exactly one of the following is true:
\begin{enumerate}
\item $(\bar m,j_\tabtwo)\in\Inv(\tabone)$
\item $(\bar m,j_\tabone)\in\Inv(\tabtwo)$
\item $j_\tabone=j_\tabtwo$.
\end{enumerate}
\end{lem}

\begin{proof}
Without loss of generality, suppose that $j_\tabone< j_\tabtwo$ and hence $\bar m$ is contained in different rows of the tableaux $\tabone$ and $\tabtwo$.  Since $\bar m$ is the smallest number in the content, it must lie at the end its respective row of $\tabone$ and $\tabtwo$.  Moreover, the content of the row indexed by $j_\tabtwo$ in $\tabone$ is strictly larger than $\bar m$ and vice versa.
If the size of row $j_\tabtwo$ is at least the size of row $j_\tabone$, then $(\bar m, j_\tabtwo)\in \Inv(\tabone)$.  Otherwise, $(\bar m, j_\tabone)\in \Inv(\tabtwo)$.
\end{proof}

Lemma~\ref{lemma.inv} induces a total ordering on the set $\RSCT_n(\beta)$ as follows.

\begin{defn}\label{D:Total_ordering}
Let $\tabone,\tabtwo\in \RSCT_n(\beta)$ and $j_\tabone,j_\tabtwo$ denote the indices of the rows containing $\bar m$ in $\tabone$ and $\tabtwo$, respectively.  First suppose $j_\tabone\neq j_\tabtwo$.  We say $\tabtwo<\tabone$ if $(\bar m, j_\tabtwo)\in\Inv(\tabone)$ and $\tabone <\tabtwo$ if $(\bar{m}, j_\tabone)\in \Inv(\tabtwo)$.  Otherwise, if $j_\tabone=j_\tabtwo$, then $\map(\tabone)$ and $\map(\tabtwo)$ have the same composition shape.  In this case, we inductively say $\tabtwo<\tabone$ if $\map(\tabtwo)<\map(\tabone)$.
\end{defn}

\begin{example}\label{ex.total_order}  Let $n=4$ and $\alpha=(2,2)$.  The total order on tableaux in $\RSCT_4(\alpha)$ is displayed below.
\[\ytableausetup{centertableaux}
\begin{ytableau} 3 & 1\\ 4 & 2\\ \end{ytableau}\,<\,
\begin{ytableau} 4 & 1\\ 3 & 2\\ \end{ytableau}\,<\,
\begin{ytableau} 2 & 1\\ 4 & 3\\ \end{ytableau}\,<\,
\begin{ytableau} 3 & 2\\ 4 & 1\\ \end{ytableau}\,<\,
\begin{ytableau} 4 & 2\\ 3 & 1\\ \end{ytableau}\,<\,
\begin{ytableau} 4 & 3\\ 2 & 1\\ \end{ytableau}
\]
\end{example}

In the next section we will associate a unique monomial to each element of $\RSCT_n(\alpha)$. We will see that the total ordering on the shifted row-strict composition tableaux defined above corresponds to the lex ordering on these monomials.


\subsection{Equivariant Springer monomials} In this section we define a collection of polynomials indexed by row-strict composition tableaux. The main purpose of defining these polynomials is to provide a combinatorial framework to study the cohomology ring $H^*_S(\B^\alpha)\simeq \ca$ in the following sections.  Indeed, the polynomials defined below will serve as an equivariant generalization of the Springer monomial basis.

\begin{defn}\label{def.p-poly}
Let $\beta$ be a composition of $m\leq n$ and  $\tabone\in\RSCT_n(\beta)$.  If $\Inv(\tabone)\neq \emptyset$, let $P_\tabone \in \C[\fs\times \ft]$ be the polynomial of degree $|\Inv(\tabone)|$ defined by,
$$P_\tabone( \mathbf{z}, \mathbf{x}):=\prod_{(i,j)\in\Inv(\tabone)} (x_i-z_j).$$
If the inversion set of $\tabone$ is empty, then define $P_\tabone = 1$. We call the collection of polynomials obtained in this way \textbf{equivariant Springer monomials}.
\end{defn}

While the $P_\tabone$ are not monomials in the traditional sense, we use the term ``monomial" since $P_\tabone$ is a product of equivariant factors $(x_i-z_j)$, a common generalization of monomials in ordinary cohomology.  We adopt the convention throughout this manuscript that each equivariant Springer monomial $P_\tabone\in \C[\fs\times \ft]$ and its image under the canonical projection map $\pi: \C[\fs\times \ft]\to \ca$ are denoted by the same symbol.  This greatly simplifies the notation below.

\begin{example}\label{ex.P-poly-short} Let $n=9$, $\beta=(2,0,3,2,1)$, and $\tabone$ as in Example~\ref{Ex:inversions}.  Then
\[
P_\tabone ( \mathbf{z}, \mathbf{x}) = (x_2-z_1)(x_2-z_3)(x_5-z_1)(x_5-z_3)(x_6-z_3).
\]
\end{example}

There is a simple inductive description of the equivariant Springer monomials, as explained in the next two paragraphs.   Suppose $\bar m$ labels a box in row $j_\tabone$ of $\tabone$ and recall that $\bar m$ must label the last box in row $j_\tabone$.  Since $\bar{m}$ is the smallest label that appears, we have
\begin{eqnarray}\label{eqn.Inv_m}
\Inv(\tabone)\cap (\{\bar{m}\}\times [k]) = \{(\bar{m}, j) \mid \beta_j>\beta_{j_\tabone} \textup{ or } \beta_j=\beta_{j_\tabone} \textup{ and } j<j_\tabone \}.
\end{eqnarray}
Denote this set by $\Inv_{\bar m}(\tabone)$.  Note in particular that $|\Inv_{\bar{m}}(\tabone)|$ is uniquely determined by the value of $j_\tabone$.

Recall the map $\map$ from~\eqref{eqn.induction}  defined by deleting the box labeled by $\bar{m}$ in $\tabone$. The Springer inversions of $\tabone$ decompose as
\begin{eqnarray*}
\Inv(\tabone) = \Inv_{\bar m} (\tabone) \sqcup \Inv(\map(\tabone)).
\end{eqnarray*}
We obtain a corresponding decomposition formula for the polynomial $P_\tabone$ given by
\begin{eqnarray}\label{eqn.Pdecomp}
P_\tabone(\mathbf{z}, \mathbf{x}) = Q_\tabone(\mathbf{z}, \mathbf{x}) P_{\map(\tabone)}(\mathbf{z}, \mathbf{x})
\end{eqnarray}
where
\[
Q_\tabone(\mathbf{z}, \mathbf{x}) := \prod_{(\bar m, j)\in \Inv_{\bar m}(\tabone)} (x_{\bar m} - z_j)
\]
if $\Inv_{\bar{m}}(\tabone)\neq \emptyset$ and $Q_\tabone = 1$ otherwise.  The next lemma shows that the decomposition formula for the polynomials $P_\tabone$ from~\eqref{eqn.Pdecomp} is compatible the recursive formula defining the Springer monomials given in equation~\eqref{eqn.GPbasis}.

\begin{lem}\label{lem.GPbasis} Let $\beta$ be a composition of $m\leq n$ and $\lambda$ denote its underlying partition shape. Then,
\begin{equation}\label{Eq:eval_to_GPbasis}
\{\ev(P_\tabone)\mid \tabone\in \RSCT_n(\beta)\}=\GP_\lambda.
\end{equation}
In particular, the set $\{\ev(P_\tabone)\mid \tabone\in \RSCT_n(\beta)\}$ only depends on $\lambda$, the underlying partition shape of $\beta$.
\end{lem}

\begin{proof} First observe that if $\beta=(\beta_1, \ldots, \beta_k)$ is a weak composition of $m\leq n$, then $\beta$ determines a unique strong composition $\tilde\beta$ obtained by deleting the parts of $\beta$ equal to 0.  If $\tabone\in \RSCT_n(\beta)$, then one obtains a unique element $\tabone' \in \RSCT_n(\beta')$ by upward justifying all rows.  It is easy to see from the definitions that $\beta$ and $\beta'$ have the same number of Springer inversions and that $\ev(P_\tabone)=\ev(P_\tabone')$.  Hence we may assume without loss of generality that $\beta$ is a strong composition of $m$.

We now proceed by (reverse) induction on $\bar{m}$, the smallest value appearing in any $\tabone\in \RSCT_n(\beta)$.  If $\bar{m}=n$ then $m=1$ and $\lambda=(1)$.  In this case, $\RSCT_n(\beta)$ contains a single element, namely the row-strict composition tableau consisting of a single box labeled by $1$. Therefore $\Inv(\tabone)=\emptyset$ and $\GP_\lambda=\{1\} = \{\ev(1)\}$, as desired.

Now suppose $\bar{m}<n$ and $\beta=(\beta_1, \ldots, \beta_k)$ has $k$ non-zero parts. Let $\sigma^{-1}$ denote the unique minimal length permutation of $k$ such that $\lambda=(\beta_{\sigma^{-1}(1)}, \ldots, \beta_{\sigma^{-1}(k)})$.  In other words, $\sigma(i)-1$ is equal to the number of $j\in [m]$ such that $\beta_j>\beta_i$ plus the number of $j\in [m]$ such that $\beta_j=\beta_i$ and $j<i$.  Combining this notation with~\eqref{eqn.Inv_m} and~\eqref{eqn.Pdecomp} implies that if $\tabone\in \RSCT_n(\beta)$ with $j_\tabone=i$ then $\ev(P_\tabone) = x_{\bar{m}}^{\sigma(i)-1} \ev(P_{\map(\tabone)})$.

Let $\beta[i]$ be the composition of $m-1$ obtained from $\beta$ by decreasing $\beta_i$ by $1$.  Note that $\tabtwo\in \RSCT_n(\beta[i])$ has content $[\bar{m}+1,n]$.  Since the map $\map$ from~\eqref{eqn.induction} is a bijection, the decomposition of $P_\tabone$ given in~\eqref{eqn.Pdecomp} now gives us,
\begin{eqnarray}\label{eqn.poly_decomp}
\{\ev(P_\tabone)\mid \tabone\in \RSCT_n(\beta)\} = \bigsqcup_{1\leq i \leq k} \{x_{\bar{m}}^{\sigma(i)-1}\ev(P_{\tabtwo})\mid \tabtwo\in \RSCT_n(\beta[i]) \}.
\end{eqnarray}
By the induction hypothesis, $\{\ev(P_{\tabtwo})\mid \tabtwo\in \RSCT_n(\beta[i])\} = \GP_{\lambda[\sigma(i)]}$ and our claim now follows directly from the recursive definition of $\GP_\lambda$ given in~\eqref{eqn.GPbasis}.
\end{proof}

\begin{example}
Consider the following tableaux in $\RSCT_6(\beta)$ for $\beta = (3,2,1)$ and $\beta = (2,1,3)$, respectively.
\[\tabone=\begin{ytableau}
4 & 3 & 2\\
6 & 5 \\
1
\end{ytableau}\quad\quad\quad  \tabtwo=\begin{ytableau}
4 & 3\\
1 \\
6 & 5& 2
\end{ytableau}
\]
Here we have $$P_\tabone=(x_1-z_1)(x_1-z_2)(x_4-z_2)\quad\text{and}\quad P_\tabtwo=(x_1-z_1)(x_1-z_3)(x_4-z_3).$$
While these polynomials are different, they correspond to the same Garsia-Procesi monomial, as $\ev(\tabone)=\ev(\tabtwo)=x_1^2x_4\in \GP_{(3,2,1)}$.
\end{example}

The next theorem tells us that the collection of equivariant Springer monomials is an $S(\fs^*)$-module basis for the equivariant cohomology ring $\ca  \simeq H_S^*(\B^\alpha)$. We study the structure coefficients of $\ca$ with respect to this basis in the next section.

\begin{thm}\label{T:P_basis}
Let $\alpha=(\alpha_1,\ldots,\alpha_k)$ be a \textup{(}strong\textup{)} composition of $n$.  The collection of equivariant Springer monomials $\{P_\tabone(\mathbf{z}, \bfx) \mid \tabone\in \RSCT_n(\alpha)\}$ is an $S(\fs^*)$-module basis of $\ca \simeq H_S^*(\B^\alpha)$.
\end{thm}

\begin{proof} The polynomials $P_\tabone(\mathbf{z}, \mathbf{x})$ are homogeneous elements of $\ca$.  Lemmas~\ref{P:equiv_basis} and~\ref{lem.GPbasis} now imply the desired result.
\end{proof}


\section{Localization and Determinant Formulas}\label{sec.determinant}

In this section, we explore algebraic properties of the equivariant Springer monomials.  The results of this section establish methods for computing the expansion of any $F\in \ca$ as an $S(\fs^*)$-linear combination of the $P_\tabone$, $\tabone\in \RSCT_n(\alpha)$.  We begin by showing that the equivariant Springer monomials satisfy upper triangular vanishing relations with respect to the total ordering defined on row-strict composition tableaux defined in the previous section.  We then use these vanishing properties to give a determinant formula for the structure coefficients in Theorem~\ref{T:P_basis_det_formula} below.


\subsection{Localization formulas}

Suppose $\alpha=(\alpha_1, \ldots, \alpha_k)$ is a strong composition of $n$.  Let $h=(h_1, \ldots, h_k)$ be a regular element of $\fs$, which we identify as a point in $\ft$, by
\begin{equation}\label{Eq:h_regular}
h=((h_1)^{\alpha_1},\ldots,(h_k)^{\alpha_k})\in \ft.
\end{equation}
The condition that $h$ be a regular element means that each of the $h_i$ are distinct.  For every $w\in W$, there is a natural localization map,
\begin{equation}\label{Eq:localization_map}
    \phi_w:S(\fs^*)\otimes S(\ft^*)\rightarrow S(\fs^*)
\end{equation}
given by $\phi_w(F(\mathbf{z}, \mathbf{x})) = F(\mathbf{z}, w\cdot \mathbf{z})$.  In other words, $\phi_w(F)(h):=F(h,w\cdot h)$ for any $h\in\fs$.  Here $W$ acts on $\fs^*$ (and the coordinates of $\mathbf{z}$) by permuting the entries; for example, if $w=[2,4,1,3]=s_1s_2$ and $h=(h_1, h_1, h_2,h_2)$ then $w\cdot h = s_1s_2\cdot h = (h_2,h_1,h_1,h_2)$.

It is easy to see that $F\in \ci(Z_\alpha)$ if and only if $\phi_w(F)\equiv 0$ for all $w\in W$.  Hence any $F\in\ca$ is uniquely determined by the collection of values $\{\phi_w(F)\mid w\in W\}$. Recall that $L$ is the Levi subgroup of $GL_n(\C)$ determined by the composition $\alpha$ and $W_L$ denotes the Weyl group of $L$.  Since $L$ is standard, the parabolic subgroup $W_L$ is generated by a subset of simple reflections. Also, since $W_L$ acts trivially on $\fs$ (because $S=Z_G(L)_0$), it suffices to consider the maps $\phi_w$ where $w\in W^L$.  Here $W^L$ denotes the set of minimal length coset representatives of $W/W_L$. Recall that each permutation $w\in W$ can be written uniquely as $w=vy$ for $v\in W^L$ and $y\in W_L$.

 We now associate a coset representative $w_\tabone\in W^L$ to each $\tabone\in \RSCT_n(\alpha)$ by constructing a vector $h_\tabone\in\ft$ which is a particular permutation of the coordinates of $h$.  Specifically, if $i$ lies in the $j$-th row of $\tabone$, then we require the $i$-th coordinate of $h_\tabone$ equal to $h_j$.  Let $w_\tabone$ to be the unique permutation in $W^L$ such that $h_\tabone=w_\tabone h$. Observe that the map from $\RSCT_n(\alpha)$ to $W^L$ given by $\tabone \mapsto w_\tabone$ is a bijection.

\begin{example}
Let $n=5$ and $\alpha=(2,1,2)$ with $\tabone$ given by:
\[\begin{ytableau}
5 & 1\\
2\\
4 & 3\\
\end{ytableau}\]
Then $h_\tabone=(h_1,h_2,h_3,h_3,h_1)$ with $w_\tabone=[1,5,2,3,4]$ (in one-line notation).  Note that in this case, $W_L = \left< s_1, s_4 \right>$ and it easy to check that $w_\tabone\in W^L$; we have only to observe that $w_\tabone(1)<w_\tabone(2)$ and $w_{\tabone}(4)<w_\tabone(5)$.    Also, in this example we have $P_\tabone=(x_2-z_1)(x_4-z_1)$ since $\Inv(\tabone) = \{ (2,1), (4,1) \}$.
\end{example}

Our next proposition says that the equivariant Springer monomials satisfy upper triangular vanishing conditions with respect to the total order on row-strict composition tableaux defined in the previous section.

\begin{prop}\label{P:upper_triangular_Pbasis}
Let $\tabtwo< \tabone \in\RSCT_n(\alpha)$.  Then the following are true:
\begin{enumerate}
\item $\phi_{w_\tabone}(P_\tabone)\neq 0$, and
\item $\phi_{w_\tabtwo}(P_\tabone)= 0.$
\end{enumerate}
\end{prop}

\begin{proof}
Fix a regular element $h\in \fs$ as in~\eqref{Eq:h_regular}.  We first prove part (1) of the proposition.  By definition, if $h_j$ is the $i$-th coordinate of $h_\tabone$, then $i$ is contained in the $j$-th row of $\tabone$.  We have
$$\phi_{w_\tabone}(P_\tabone)(h)=P_\tabone(h,h_\tabone)=\prod_{(i,j)\in\Inv(T)} ((h_\tabone)_i-h_j).$$
Note that if $(i,j)\in\Inv(\tabone)$, then $i$ cannot be contained in the $j$-th row of $\tabone$.  Hence $(h_\tabone)_i\neq h_j$ for all $(i,j)\in\Inv(\tabone)$ and $\phi_{w_\tabone}(P_\tabone)(h)\neq 0$ as claimed.

We now prove part (2).  Indeed, we have
$$
\phi_{w_\tabtwo}(P_\tabone)(h)=P_\tabone(h,h_\tabtwo)=\prod_{(i,j)\in\Inv(T)} ((h_\tabtwo)_i-h_j).
$$
Since $\tabtwo<\tabone$, there exists $(i,j)\in \Inv(\tabone)$ such that the content of $j$-th row of $\tabtwo$ contains $i$.  This implies that $(h_\tabtwo)_i=h_j$ and hence $\phi_{w_\tabtwo}(P_\tabone)(h)=0$. Since $h\in\fs$ is an arbitrary regular element, we have $\phi_{w_\tabtwo}(P_\tabone)= 0$ in~$S^*(\fs)$.
\end{proof}

\begin{rem} A alternative proof of Theorem~\ref{T:P_basis} from the previous section can be given using Proposition~\ref{P:upper_triangular_Pbasis} as follows. Note that one can establish the fact that $\{P_\tabone\mid \tabone \in \RSCT_n(\alpha)\}$ is an $S(\fs^*)$-linearly independent set by using the vanishing conditions of Proposition~\ref{P:upper_triangular_Pbasis}.  Furthermore, the number of polynomials in  $\{P_\tabone\mid \tabone \in \RSCT_n(\alpha)\}$ of degree $k$ is precisely $H^{2k}(\B^\alpha)$ by Lemma~\ref{lem.GPbasis}. Thus $\{P_\tabone\mid \tabone \in \RSCT_n(\alpha)\}$ is an $S(\fs^*)$-basis of $\ca$ by Proposition 18 of \cite{HT11}.
\end{rem}

We conclude with a detailed example.

\begin{example}\label{ex.(2,2)} Let $n=4$ and $\alpha=(2,2)$.  A table of $P_\tabone$, $w_\tabone$ and $h_\tabone$ for all elements $\tabone\in\RSCT_4(\alpha)$ is displayed in Figure 1 below.
The matrix $[\phi_{w_\tabtwo}(P_\tabone)]_{(\tabone,\tabtwo)\in \RSCT_4(\alpha)^2}$ written with respect to the total ordering on $\RSCT_4(\alpha)$ given in Example \ref{ex.total_order} is:
$$\begin{bmatrix}
\quad 1\quad & 1 & 1 & 1 & 1 & 1\\
0 & z_2-z_1 & z_2-z_1 & 0 & z_2-z_1 & 0\\
0 & 0 & z_1-z_2 & z_1-z_2 & z_1-z_2 & 0\\
0 & 0 & 0 & z_2-z_1 & z_2-z_1 & z_2-z_1\\
0 & 0 & 0 & 0 & (z_2-z_1)^2 & 0\\
0 & 0 & 0 & 0 & 0 & (z_2-z_1)^2\\
\end{bmatrix}.$$
Proposition \ref{P:upper_triangular_Pbasis} implies this matrix is always upper triangular with respect to the total ordering in Definition \ref{D:Total_ordering} with non-vanishing polynomials in $S(\fs^*)$ on the diagonal.
\begin{figure}[h]
\[
\begin{tabular}{cccc}
$\tabone$& $P_\tabone$ & $w_\tabone$ & $h_\tabone$ \\ \hline
\multirow{2}{*}{\begin{ytableau} 3 & 1\\ 4 & 2\\ \end{ytableau}}& \multirow{2}{*}{1} &  \multirow{2}{*}{$[1,3,2,4]$} \,&  \multirow{2}{*}{$(h_1, h_2, h_1, h_2)$}\\
& & & \\
\multirow{2}{*}{\begin{ytableau} 4 & 1\\ 3 & 2\\ \end{ytableau}}&  \multirow{2}{*}{$x_3-z_1$} &  \multirow{2}{*}{$[1,4,2,3]$}  \, &  \multirow{2}{*}{$(h_1, h_2, h_2, h_1)$}\\
& & & \\
 \multirow{2}{*}{\begin{ytableau} 2 & 1\\ 4 & 3\\ \end{ytableau}}&  \multirow{2}{*}{$x_2-z_2$} &  \multirow{2}{*}{$[1,2,3,4]$} \, &  \multirow{2}{*}{$(h_1, h_1, h_2, h_2)$}\\
& & & \\
 \multirow{2}{*}{\begin{ytableau} 3 & 2\\ 4 & 1\\ \end{ytableau}}&  \multirow{2}{*}{$x_1-z_1$} &  \multirow{2}{*}{$[2,3,1,4]$}  \,&  \multirow{2}{*}{$(h_2, h_1, h_1, h_2)$}\\
& & & \\
 \multirow{2}{*}{\begin{ytableau} 4 & 2\\ 3 & 1\\ \end{ytableau}} &  \,\multirow{2}{*}{ $(x_1-z_1)(x_3-z_1)$} \, \, &  \multirow{2}{*}{$[2,4,1,3]$}  \,&  \multirow{2}{*}{$(h_2, h_1, h_2, h_1)$}\\
& & & \\
 \multirow{2}{*}{\begin{ytableau} 4 & 3\\ 2 & 1\\ \end{ytableau}}&  \multirow{2}{*}{$(x_1-z_1)(x_2-z_1)$} &  \multirow{2}{*}{$[3,4,1,2]$} \, &  \multirow{2}{*}{$(h_2, h_2, h_1, h_1)$}\\
 &&&
\end{tabular}
\]
\begin{caption}{Equivariant Springer monomials for $n=4$ and $\alpha=(2,2)$.}
\end{caption}
\label{fig1}
\end{figure}
\end{example}


\subsection{Structure constants for the equivariant Springer monomials}

We now present a determinant formula for calculating the structure coefficients of the expansion of $F\in \ca$ in the basis of equivariant Springer monomials.  For these calculations, we work in the algebra
$$Q(\fs^*)\otimes_{S(\fs^*)}\ca$$
where the $Q(\fs^*)$ denotes the field of fractions of $S(\fs^*)$.  We index the set
$$\RSCT_n(\alpha)=\{\tabone_1<\cdots < \tabone_N\}$$
by the total ordering given in Definition~\ref{D:Total_ordering} where $N=|\RSCT_n(\alpha)|=|W^L|$.

For notational and computational simplicity, let $P_i:=P_{\tabone_i}$ and $w_i:=w_{\tabone_i}$.  Given any $F\in\ca$, we write
\begin{equation}\label{Eq:f_exp_to_P}
F=\sum_{k=1}^N C_k\, P_k
\end{equation}
for some coefficients $C_k\in S(\fs^*)$.  Define vectors
$$\textbf{c}:=[C_1,\ldots,C_N] \quad\text{and}\quad \textbf{v}:=[\phi_{w_1}(f),\ldots, \phi_{w_N}(f)],$$
and the matrix
$$\overline{P}:=\left[\phi_{w_j}(P_i)\right]_1^N.$$
Note that $\overline{P}$ was computed for $n=4$ and $\alpha=(2,2)$ in Example~\ref{ex.(2,2)}.  Equation~\eqref{Eq:f_exp_to_P} implies $\textbf{c}\cdot \overline{P}=\textbf{v}$. Proposition~\ref{P:upper_triangular_Pbasis} tells us that $\overline{P}$ is an upper triangular matrix with nonzero diagonal entries, and is therefore invertible as a matrix with entries in $Q(\fs^*)$.   Hence
\begin{equation}\label{Eq:inverse_coefficent_formula}\textbf{c}=\textbf{v}\cdot \overline{P}^{-1}.\end{equation}
Our next theorem uses this equation to prove that each coefficient $C_k$ is the determinant of some matrix with entries determined by $\mathbf{v}$ and $\phi_{w_j}(P_i)$.  Normalize the polynomials $P_i$ by defining $Q_i:=\frac{1}{\phi_{w_i}(P_i)}\cdot P_i.$ Note that this definition makes sense, since $\phi_{w_i}(P_i)\neq 0$ for all $i$ by Proposition~\ref{P:upper_triangular_Pbasis}.

\begin{thm}\label{T:P_basis_det_formula}
Suppose $F\in \ca$ and define
\[
a(i,j):=
\begin{cases}
    \phi_{w_j}(F) & \text{for $i=0$}\\
    \phi_{w_j}(Q_i) & \text{for $i>0$}.
\end{cases}
\]
Write
\begin{equation}\label{Eq:f_exp_to_Q}
 F=\sum_{k=1}^N D_k\, Q_k.
\end{equation}
Then $D_k=(-1)^{k-1}\det\left[a(i,j+1)\right]_0^{k-1}$. In particular, the coefficients for $F$ appearing in~\eqref{Eq:f_exp_to_P} are
\[
C_k = \frac{(-1)^{k-1}}{\phi_{w_i}(P_i)}\det\left[a(i,j+1)\right]_0^{k-1}
\]
for all $1\leq k\leq N$.
\end{thm}

\begin{proof}
Define the matrix  $A_k:=[a(i,j+1)]_0^{k-1}$ and let $A_{\ell,k}$ denote the submatrix of $A_k$ obtained by removing the $\ell$-th row and $k$-th column.  Applying Proposition~\ref{P:upper_triangular_Pbasis}, we observe that $a(i,i)=1$ for all $i\geq 1$ and $a(i,j)=0$ if $i>j$.  This implies $\det(A_{1,k})=1$ and
\begin{equation}\label{Eq:determinant}
\det(A_{\ell,k})=\det(A_{\ell-1})
\end{equation}
for $\ell\geq 2$.  We prove the theorem by induction on $k$.  When $k=1$, applying $\phi_{w_1}$ to both sides of Equation \ref{Eq:f_exp_to_Q} gives $D_1=a(0,1)$, as desired.  Now suppose for all $\ell<k$, we have that
\begin{equation}\label{eqn1.det}
D_\ell=(-1)^{\ell-1}\det(A_{\ell}).
\end{equation}
We now apply the localization map $\phi_{w_k}$ to both sides of Equation \eqref{Eq:f_exp_to_Q}.  Solving for $D_k$ and applying Equations \eqref{Eq:determinant} and \eqref{eqn1.det} yields
\begin{align*}
D_k&=a(0,k)-\sum_{\ell=1}^{k-1} D_\ell\, a(\ell,k)\\
&=a(0,k)+\sum_{\ell=1}^{k-1}(-1)^{\ell}\det(A_{\ell})\, a(\ell,k).\\
&=a(0,k)+\sum_{\ell=1}^{k-1}(-1)^{\ell}\det(A_{\ell+1,k})\, a(\ell,k).\\
&=(-1)^{k-1}\det(A_k),
\end{align*}
proving the theorem.
\end{proof}

This theorem provides us with the computational tools to expand any polynomial of $\ca$ in the basis of equivariant Springer monomials.  It follows immediately that we can compute the expansion of any polynomial in $\ca_0\simeq H^*(\B^\alpha)$ in the Springer monomial basis by simply applying the evaluation map $\ev:\ca\rightarrow \ca_0$.  We use these results in the next section to study the images of monomials and Schubert polynomials in $H^*(\B^\alpha)$.

\begin{example} Let $n=4$ and $\alpha=(2,2)$.  The polynomials $P_\tabone$ for $\tabone\in \RSCT_4(\alpha)$ are computed in Example~\ref{ex.(2,2)} (see Figure 1 also).  In this case $N=6$ and the total order on $\RSCT_4(\alpha)$ is as in Example~\ref{ex.total_order}, so the rows of the table in Figure 1 list the polynomials in order: $P_1,\cdots, P_6$, from top to bottom. We compute the expansion of
\[
F(\mathbf{z},\mathbf{x}) = x_1+x_2+x_3 - 2z_1+z_2
\]
using the determinant formula of Theorem~\ref{T:P_basis_det_formula}. The reader may note that $F$ is the image of the double Schubert polynomial $\mathfrak{S}_{s_3}(\mathbf{y}, \mathbf{x})\in \C[\ft\times \ft]$ under the map $i^*: \C[\ft \times \ft]\to \C[\fs\times \ft]$.
The matrix $[a(i,j+1)]_0^{5}$ from Theorem \ref{T:P_basis_det_formula} is given by:
$$\begin{bmatrix}
\quad 0\quad & z_2-z_1 & \quad 0\quad  &\quad 0\quad & z_2-z_1 & z_2-z_1\\
1 & 1 & 1 & 1 & 1 & 1\\
0 & 1 & 1 & 0 & 1 & 0\\
0 & 0 & 1 & 1 & 1 & 0\\
0 & 0 & 0 & 1 & 1 & 1\\
0 & 0 & 0 & 0 & 1 & 0\\
\end{bmatrix}.$$
Where the first row is the vector $\textbf{v}=[\phi_{w_1}(F),\ldots, \phi_{w_6}(F)]$ with the rest of the matrix coming from first five rows of the matrix in Example \ref{ex.(2,2)} (normalized to $\phi_{w_j}(Q_i)$).
If $F=\sum_{i=1}^6 D_k\, Q_k$, then Theorem \ref{T:P_basis_det_formula} says the coefficients $D_i$ are given by the upper-left minors (with a sign) yielding:
$$\textbf{d}:=[D_1,\ldots,D_6]=[0,z_2-z_1, z_1-z_2, z_2-z_1,0,0].$$
This implies
$$\textbf{c}:=[C_1,\ldots,C_6]=[0,1, 1, 1,0,0]$$
and hence $F=P_2+P_3+P_4.$  Note that we can also compute $\textbf{c}$ by using the equation $\textbf{c}=\textbf{v}\cdot\overline{P}^{-1}$ with
$$\overline{P}^{-1} = \begin{bmatrix}
\quad 1\quad & (z_1-z_2)^{-1} & 0 & (z_1-z_2)^{-1} & (z_2-z_1)^{-2} & 0\\
0 & (z_2-z_1)^{-1} & (z_2-z_1)^{-1} & (z_2-z_1)^{-1} & -(z_2-z_1)^{-2} & -(z_2-z_1)^{-2}\\
0 & 0 & (z_1-z_2)^{-1} & (z_1-z_2)^{-1} & 0 & (z_2-z_1)^{-2}\\
0 & 0 & 0 & (z_2-z_1)^{-1} & -(z_2-z_1)^{-2} & -(z_2-z_1)^{-2}\\
0 & 0 & 0 & 0 & (z_2-z_1)^{-2} & 0\\
0 & 0 & 0 & 0 & 0 & (z_2-z_1)^{-2}\\
\end{bmatrix}$$
the inverse of the matrix from Example~\ref{ex.(2,2)}.
\end{example}

\begin{rem}
If $F$ is the image of a double Schubert polynomial $\mathfrak{S}_{w}(\mathbf{y}, \mathbf{x})\in \C[\ft\times \ft]$, then the vector $\textbf{v}=[\phi_{w_1}(F),\ldots,\phi_{w_N}(F)]$ can be computed directly using Billey's localization formula (also called the  Andersen--Jantzen--Soergel formula) given in \cite[Theorem 3]{Bi96}.
\end{rem}


\section{Monomials and Schubert polynomials}\label{sec.monomial-Schubert}
In this section, we study the images of the Schubert polynomials $\mathfrak{S}_w(\bfx)$ under the map $\pi_0:\C[\ft] \to \ca_0$.   We use Theorem \ref{T:P_basis_det_formula} to identify an explicit collection of permutations $W(\alpha)\subset W$ for which the set $\{\pi_0(\mathfrak{S}_w(\bfx))\ |\ w\in W(\alpha)\}$ is a basis of $\ca_0 \simeq H^*(\B^\alpha)$.  This result is stated in Theorem~\ref{T:Schubert_Basis}.  We obtain an analogous statement for equivariant cohomology in Corollary~\ref{cor.equiv-main-thm}.  Our analysis generalizes work of Harada--Tymoczko \cite{HT17} and Harada--Dewitt \cite{DH} in the sense that Corollary~\ref{cor.equiv-main-thm} implies the existence of an explicit module basis for $H_S^*(\B^\alpha)$ constructed by playing poset pinball.

We prove Theorem~\ref{T:Schubert_Basis} in two steps.  First, we use the expansion formula of Theorem~\ref{T:P_basis_det_formula} to prove that the Springer monomial basis $\GP_\lambda$ of $H^*(\B^\alpha)$ defined in~\eqref{eqn.GPbasis} above is upper-triangular in an appropriate sense.  In particular, we study the expansion of any monomial in $\ca_0$ with respect to the Springer monomial basis.  Since each Schubert polynomial is a sum of monomials, we are then able to leverage our results for monomials to prove the desired result for Schubert polynomials.  More specifically, we prove that the transition matrix from $\{\pi_0(\mathfrak{S}_w)\mid w\in W(\alpha)\}$ to $\GP_\lambda$ is invertible.

To begin, recall the commutative diagram from~\eqref{eqn.cd}.  In particular, recall that
$\ca \simeq\C[\ft\times\ft]/\ci(Z_\alpha)$ and $\ca_0 \simeq \C[\ft]/\ev(\ci(Z_\alpha))$ and the maps $\pi$ and $\pi_0$ denote the canonical projection maps.


\subsection{Monomials}\label{sec.monomial} The first class of polynomials we study are monomials in the ring $\C[\ft]\simeq \C[x_1,\ldots,x_{n-1}]$.  Monomials in $\C[\ft]$ are indexed by weak compositions $\delta=(\delta_1,\ldots,\delta_{n-1})$ under the exponent identification
$$\delta\mapsto \bfx^\delta:=x_1^{\delta_1}\cdots x_{n-1}^{\delta_{n-1}}.$$
We impose the lexicographical total ordering on monomials.  In other words, $\bfx^{\gamma}<\bfx^\delta$ if and only if $\gamma_k<\delta_k$ where $k$ denotes the smallest index where the entries of the compositions $\gamma$ and $\delta$ differ.

If $\tabone\in \RSCT_n(\alpha)$, then $\ev(P_\tabone)$ is a monomial in $\C[\ft]$.  Hence we define the notation
$$\bfx^\tabone:=\ev(P_\tabone).$$
By Lemma~\ref{lem.GPbasis} the set of all monomials obtained in this way is precisely the set of Springer monomials $\GP_\lambda$ where $\lambda$ is the underlying partition shape of $\alpha$.  Recall that, by convention, since $\bfx^\tabone \in \GP_\lambda$ we also write $\bfx^\tabone$ to denote the image of the monomial $\bfx^\tabone$ in $\ca_0$ under $\pi_0$.  Observe that if $\gamma$ is the associated exponent composition of $\bfx^\tabone$, then $\gamma_i$ is simply the number of inversions in $\Inv(\tabone)$ whose first factor is $i$.  Hence we will call the composition $\gamma$ the \textbf{inversion vector} of $\tabone$.  For $\tabone$ as in Example~\ref{Ex:inversions}, the inversion vector is $\gamma=(0,2,0,0,2,1,0,0)$ and $\bfx^\tabone=x_2^2x_5^2x_6$.  The next lemma follows immediately from the definition of the total order on $\RSCT_n(\alpha)$.

\begin{lem}\label{L:monomial-tab_order} Let $\tabone,\tabtwo\in \RSCT_n(\alpha)$. Then $\tabtwo<\tabone$ as row strict composition tableaux \textup{(}c.f.~Definition~\ref{D:Total_ordering}\textup{)} if and only if $\bfx^\tabtwo<\bfx^\tabone$ as monomials in $\C[\ft]$.
\end{lem}

Lemma \ref{L:monomial-tab_order} implies that the vanishing property given in Proposition \ref{P:upper_triangular_Pbasis} is, in some way, compatible with the total ordering on all monomials.  To make this compatibility precise, for each monomial $\bfx^{\delta}\in \C[\ft]$ we construct a polynomial $P_{\delta}(\mathbf{z},\mathbf{x})\in\C[\fs\times\ft]$ such that $\ev(P_\delta) = \bfx^\delta$.  This polynomial serves as an analogue of $P_\tabone(\mathbf{z},\mathbf{x})$ for $\bfx^\tabone$ when $\tabone\in\RSCT_n(\alpha)$.

Let $\delta=(\delta_1, \ldots, \delta_{n-1})$ be a composition of $n$. If $\delta$ is the inversion vector for some $\tabone\in \RSCT_n(\alpha)$, then set $P_\delta = P_\tabone$.  Otherwise,  by Lemma \ref{L:monomial-tab_order} there is a unique maximal $\tabone\in\RSCT_n(\alpha)$ such that $\bfx^{\tabone}< \bfx^{\delta}$.  Let $\gamma=(\gamma_1,\ldots,\gamma_{n-1})$ denote the inversion vector of~$\tabone$.  By definition of the total ordering on monomials, there exists an index $k$ such that $\gamma_i=\delta_i$ if $i<k$ and $\gamma_k<\delta_k$.   Let $\Inv_{\leq k}(\tabone):=\{(i,j)\in\Inv(\tabone)\ |\ i\leq k\}$.  We define the polynomial $P_\delta(\mathbf{z}, \mathbf{x})\in\C[\fs\times \ft]$  by
\begin{equation}\label{Eq:P_delta}
    P_\delta(\mathbf{z},\mathbf{x}):=\bfx^{\delta'}\cdot\,(x_k-z_{j'})\,\cdot\prod_{(i,j)\in \Inv_{\leq k}(\tabone)}(x_i-z_j)
\end{equation}
where $j'$ denotes the index of the row containing $k$ in $\tabone$ and the composition $\delta'$ is defined by
\[
\delta_i':=\begin{cases} 0 &\text{if $i<k$} \\ \delta_k-\gamma_k-1 &\text{if $i=k$} \\ \delta_i &\text{if $i>k$.}\end{cases}
\]

The following example illustrates the construction.

\begin{example}\label{ex.8}
Let $n=8$ and $\alpha=(2,3,1,2)$.  Consider the monomial $\bfx^\delta=x_3^2x_5^4x_7$ with $\delta=(0,0,2,0,4,0,1,0)$.  The maximal $\tabone\in\RSCT_n(\alpha)$ with $\bfx^\tabone<\bfx^\delta$ is
\[\begin{ytableau}
3 & 2\\
6 & 4 & 1\\
5 \\
8 & 7\\
\end{ytableau}\]
with $\bfx^\tabone=x_3^2x_5^2x_6$ and $\gamma=(0,0,2,0,2,1,0,0)$.  Note that the compositions $\delta$ and $\gamma$ agree in the first four entries with $\delta_5>\gamma_5$ so $k=5$ and $j'=3$ in this case.  We have $\Inv(\tabone)=\{(3,2), (3,4), (5,2),(5,4),(6,4)\}$, so
$$P_\delta=\underbrace{(x_5x_7)}_{\bfx^{\delta'}} \cdot (x_5-z_3)\cdot \underbrace{(x_3-z_2)(x_3-z_4)(x_5-z_2)(x_5-z_4)}_{(x_i-z_j) \textup{ with } (i,j)\in\Inv_{\leq 5}(\tabone)}.$$
\end{example}

The next lemma is a technical result proving the key computational properties of $P_\delta$.

\begin{lem}\label{L:monomial-vanishing}
Let $\bfx^{\delta}\in \C[\ft]$.  Then we obtain the following:
\begin{enumerate}
\item $\ev(P_\delta)=\bfx^\delta$, and
\item $\phi_{w_\tabtwo}(P_\delta)=0$ for all $\tabtwo\in \RSCT_n(\alpha)$ such that $\bfx^\tabtwo<\bfx^\delta$, where $\phi_{w_\tabtwo}$ is the localization map defined in~\eqref{Eq:localization_map} above.
\end{enumerate}
\end{lem}

\begin{proof}
It easy to see by construction that $\ev(P_\delta)=\bfx^{\delta}$ which proves (1). If $\delta$ is the inversion vector for some $\tabone\in \RSCT_n(\alpha)$, then (2) is an immediate consequence of Proposition~\ref{P:upper_triangular_Pbasis}.  Thus we have only to prove (2) in the case that $\delta$ is not the inversion vector for some row strict composition tableau.  Let $\tabone$ be the maximal element of $\RSCT_n(\alpha)$ such that $x^\tabone<x^\delta$.

First observe that if $\tabtwo=\tabone$, then $\phi_{w_\tabtwo}(P_\delta)=0$ since the factor $(x_k-z_{j'})$ in $P_\delta$ evaluates to zero on any $(h, w_\tabone\cdot h)$ with $h\in \fs$.  Now suppose $\tabtwo<\tabone$.  By definition of the total order on $\RSCT_n(\alpha)$, there exists $(i,j)\in \Inv(\tabone)$ such that the content of the $j$-th row of $\tabtwo$ contains $i$.  Furthermore, we have that the numbers $i+1, \ldots, n$ appear in the same rows (and the same exact position) of $\tabone$ and $\tabtwo$.  If $(i,j)\in\Inv_{\leq k}(\tabone)$, i.e.~if $i\leq k$, then $\phi_{w_\tabtwo}(P_\delta)=0$.  Otherwise, if $i>k$ then the tableaux $\tabtwo$ and $\tabone$ must contain $k$ in the same row. This implies $\phi_{w_\tabtwo}(P_\delta)=0$ due to the factor $(x_k-z_{j'})$ again evaluating to zero.
\end{proof}

The following proposition tells us that the expansion of $\pi_0(\bfx^\delta)$ in the Springer monomial basis contains only monomials $x^\tabone$ for $\tabone\in \RSCT_n(\alpha)$ such that $\bfx^\delta \leq \bfx^\tabone$.  This is what we mean when we say that the Springer monomial basis is compatible with the total ordering on all monomials.
Note that the proposition is also true if we impose the graded lexicographical order on monomials since $\pi_0$ is a graded map.

\begin{prop}\label{P:monomial_expansion}
Let $\bfx^{\delta}\in \C[\ft]$.  Then
$$\pi_0(\bfx^{\delta})=\sum c_\tabone\, \bfx^\tabone$$
where the sum is over all $\tabone\in\RSCT_n(\alpha)$ such that $\bfx^\tabone\geq \bfx^{\delta}$.  In other words, if $\bfx^\tabone<\bfx^{\delta}$, then $c_\tabone=0$.
\end{prop}

\begin{proof}
Let $\bfx^{\delta}\in \C[\ft]$ and note that  if $\bfx^\delta = \bfx^\tabone$ for some $\tabone\in \RSCT_n(\alpha)$, then the proposition is trivial.  We therefore assume that $\bfx^\delta \neq \bfx^\tabone$ for any $\tabone\in \RSCT_n(\alpha)$, i.e., that $\delta$ is not the inversion vector for any row strict composition tableaux of shape $\alpha$.  Consider the polynomial $P_\delta\in \C[\fs\times\ft]$ as defined in equation~\eqref{Eq:P_delta} and write
\[
\pi(P_\delta)=\sum_{\tabone'} C_{\tabone'}\, P_{\tabone'}\in \ca.
\]
Let $\tabone\in\RSCT_n(\alpha)$ be the unique maximal tableau for which $\bfx^{\tabone}< \bfx^{\delta}$.  Theorem~\ref{T:P_basis_det_formula} and Lemma~\ref{L:monomial-vanishing} together imply $C_\tabtwo=0$ for all $\tabtwo\leq \tabone$. (Note that this fact also follows from equation~\eqref{Eq:inverse_coefficent_formula}).   Again by Lemma~\ref{L:monomial-vanishing}, we have $\ev(P_\delta)=\bfx^\delta$ and hence $c_\tabtwo=\ev(C_\tabtwo)=0$ for all $\tabtwo\leq \tabone$.
\end{proof}

We demonstrate Proposition~\ref{P:monomial_expansion} with an example.

\begin{example}\label{Ex:P_delta_expansion}
Let $n=6$, $\alpha=(3,3)$, and $\delta=(0,1,1,0,1)$.  The $\bfx^{\delta}=x_2x_3x_5$ and $P_\delta=(x_2-z_2)(x_3-x_2)(x_5-z_2)$. The tableaux
\[
\tabone=\begin{ytableau} 3 & 2 & 1\\ 6 & 5 & 4\\ \end{ytableau}
\]
is the unique maximal element of $\RSCT_6(\alpha)$ such that $\bfx^{\tabone}< \bfx^{\delta}$.  In this case, we have $\bfx^{\tabone}=x_2x_3$.  If we write $\pi(P_\delta)=\sum_{\tabone'} C_{\tabone'}\, P_{\tabone'}$, then the coefficients $C_{\tabone'}$ can be computed using Theorem \ref{T:P_basis_det_formula}.  The nonzero coefficients are listed in the table appearing in Figure 2.
From this information, we immediately get that
$$\pi_0(\bfx^\delta)=-(x_1x_3x_5+x_1x_2x_5+x_1x_2x_3).$$
If we label $\RSCT_6(\alpha)=\{\tabone_1<\cdots< \tabone_{20}\}$ with respect to the total order, then $\tabone=\tabone_{10}$ and the set of tableaux corresponding to nonzero coefficients are:
$$\{\tabone_{11}, \tabone_{12},\tabone_{14},\underline{\tabone_{15}},\tabone_{17},\underline{\tabone_{18}},\underline{\tabone_{20}}\}.$$
The underlined tableaux correspond to nonzero constant coefficients.

\begin{figure}[h]
\[
\begin{tabular}{c|ccccccc}
$\tabone'$ &
\begin{ytableau} 5 & 3 & 2\\ 6 & 4 & 1\\ \end{ytableau}&
\begin{ytableau} 6 & 3 & 2\\ 5 & 4 & 1\\ \end{ytableau}&
\begin{ytableau} 5 & 4 & 2\\ 6 & 3 & 1\\ \end{ytableau}&
\begin{ytableau} 6 & 4 & 2\\ 5 & 3 & 1\\ \end{ytableau}&
\begin{ytableau} 5 & 4 & 3\\ 6 & 2 & 1\\ \end{ytableau}&
\begin{ytableau} 6 & 4 & 3\\ 5 & 2 & 1\\ \end{ytableau}&
\begin{ytableau} 6 & 5 & 4\\ 3 & 2 & 1\\ \end{ytableau}\\ \\ \hline \\
$\bfx^{\tabone'}$& $x_1$& $x_1x_5$ & $x_1x_3$ & $x_1x_3x_5$ & $x_1x_2$ & $x_1x_2x_5$ & $x_1x_2x_3$ \\ \\
$C_{\tabone'}$ & $-(z_2-z_1)^2$& $(z_2-z_1)$ & $(z_2-z_1)$ & $-1$ & $(z_2-z_1)$ & $-1$ & $-1$ \\
\end{tabular}
\]
\begin{caption}{Coefficients of $\pi(P_\delta)$ for $n=6$, $\alpha=(3,3)$, and $\delta = (0,1,1,0,1)$.}
\end{caption}
\end{figure}
\end{example}

One immediate consequence of Proposition \ref{P:monomial_expansion} is the following.

\begin{cor}Let $\bfx^{\delta}\in \C[\ft]$ and let $F\in \ca$ such that $\ev(F)=\bfx^{\delta}$.  Write
$$F=\sum_{\tabone\in\RSCT_n(\alpha)} C_\tabone\, P_\tabone.$$
If $C_\tabone\neq 0$ and $\mathbf{x}^{\tabone}<\mathbf{x}^\delta$, then $k=\deg(\mathbf{x}^{\tabone})<\deg(\mathbf{x}^{\delta})=m$ and $C_\tabone \in S^{m-k}(\fs^*)$.
\end{cor}


\subsection{Schubert polynomials}

The set of Schubert polynomials $\{\mathfrak{S}_w(\bfx)\mid w\in W\}$ in $\C[\ft]$ is an important collection of polynomials.  Note that the map $\pi_0:\C[\ft]\rightarrow \ca_0$ factors through $\phi^*_0:\ca'_0\rightarrow \ca_0$ where $\ca_0'\simeq H^*(\B)$ (see Theorem \ref{T:KP_theorem}) and hence $\mathfrak{S}_w(\bfx)$ may be viewed as a polynomial in $\ca_0'$.  It is widely known that Schubert polynomials are representatives for the Schubert classes in $H^*(\B)$ and form a basis of the cohomology ring.  The main result of this section is Theorem \ref{T:Schubert_Basis} which states there is a natural subset $W(\alpha)\subseteq W$ such that the set of images $\{\pi_0(\mathfrak{S}_w)\mid w\in W(\alpha)\}$ form a basis for the cohomology of the Springer fiber $\ca_0 \simeq H^*(\B^\alpha)$. Corollary~\ref{cor.equiv-main-thm} in this section proves an equivariant version of this statement and generalizes results of Harada--Tymoczko \cite{HT17} and Harada--Dewitt \cite{DH}.

Given a permutation $w\in W$, we recall that the \textbf{inversion set} of $w$ is
\[
\Inv(w):= \{(i<j) \mid w(i)>w(j)\}.
\]
Recall that the \textbf{length} of a permutation $w$ is $\ell(w) = |\Inv(w)|$.  The \textbf{Lehmer code} of $w$ is defined as the sequence $(\gamma_1(w),\gamma_2(w), \ldots, \gamma_{n-1}(w), \gamma_n(w))$ where $\gamma_k(w)$ denotes the number of inversions of $w$ of the form $(k,j)$ for some $j$.  Given any permutation $w$ and $k\in [n]$, it is clear that $0\leq \gamma_k(w) \leq n-k$.  On the other hand, given a sequence of nonnegative integers $(\gamma_1, \gamma_2, \ldots, \gamma_{n-1}, \gamma_n)$ such that $0\leq \gamma_k\leq n-k$, the following well known lemma defines an explicit permutation $w$ with Lehmer code $(\gamma_1, \gamma_2, \ldots, \gamma_{n-1}, \gamma_n)$. See, for example, \cite[Ch.~2]{Bjorner-Brenti} for a proof.

\begin{lem}\label{lem.inversions} Suppose  $(\gamma_1, \gamma_2, \ldots, \gamma_{n-1},\gamma_n)$ is a sequence of nonnegative integers such that $\gamma_k \leq n-k$ for $k\in [n]$.  For each such $k$, define
\[
w_k:=s_{k+\gamma_k-1}s_{k+\gamma_k-2}\ldots  s_{k+1}s_k
\]
if $\gamma_k\neq 0$, and $w_k=e$ if $\gamma_k=0$.  Then $w=w_1w_2\ldots w_{n-1}\in W$ has Lehmer code $(\gamma_1, \gamma_2, \ldots, \gamma_{n-1},\gamma_n=0)$, and $w$ is unique with respect to this property.
\end{lem}

We now describe the set $W(\alpha)\subseteq W$.  This subset is analogous to the set of \textit{Schubert points} defined by the first author and Tymoczko in \cite{PT19}, although our conventions differ, as discussed in Remark~\ref{rem.conventions} above.  To any $\tabone \in\RSCT_n(\alpha)$ we define $u_\tabone$ to be the unique permutation (as defined in Lemma~\ref{lem.inversions}) such that the inversion vector of $\tabone$ equals the Lehmer code of $u_\tabone$.  Define
$$W(\alpha):=\{u_\tabone \mid \tabone \in \RSCT_n(\alpha)\}.$$
This collection of permutations has the property that the number of $w\in W(\alpha)$ with Bruhat length $k$ is precisely $\dim(H^{2k}(\B^\alpha))$.  Thus the set $W(\alpha)$ is the output of a successful game of \textit{Betti pinball} in the sense of~\cite{HT17}.

\begin{example} Let $n=8$ and $\alpha=(2,3,1,2)$. Take $\tabone\in \RSCT_n(\alpha)$ to be as in Example~\ref{ex.8}, and recall that $\tabone$ has exponent vector $\gamma = (0,0,2,0,2,0,1,0)$.  Applying Lemma~\ref{lem.inversions} we have $u_\tabone= s_4s_3 s_6s_5 s_7$ (where $w_{3}=s_4s_3$, $w_5=s_6s_5$, $w_7 = s_7$).  The Lehmer code of $u_{\tabone}$ is~$\gamma$.
\end{example}
We can now state the main theorem of this section.

\begin{thm}\label{T:Schubert_Basis}
The set $\{\pi_0(\mathfrak{S}_{w}(\bfx)) \mid w\in W(\alpha)\}$ forms an additive basis of $H^*(\B^\alpha)$.
\end{thm}

Before we prove the theorem, we review the definition Schubert polynomials given by Lascoux and Sch\"{u}tzenberger in \cite{LaSc81} and prove a key property about their monomial expansions.  First recall Newton's divided difference operator $\partial_i:\C[\ft]\rightarrow \C[\ft]$ defined as:
$$\partial_i(f):=\frac{f-s_i(f)}{x_i-x_{i+1}}$$
where $s_i(f)$ is the polynomial obtained by swapping the variables $x_i$ and $x_{i+1}$ in $f$.  The Schubert polynomials are defined recursively by first setting $$\mathfrak{S}_{w_0}(\bfx):=x_1^{n-1}x_2^{n-2}\cdots x_{n-1}$$ where $w_0=[n,n-1,\ldots,1]$ denotes the longest permutation in $W$ and then defining
\begin{eqnarray}\label{eqn.schubertdef}
\mathfrak{S}_{w}(\bfx):=\partial_i(\mathfrak{S}_{ws_i}(\bfx))
\end{eqnarray}
if $\ell(ws_i)=\ell(w)+1$.
Since the divided difference operators $\partial_i$ satisfy the braid relations on $W$, \eqref{eqn.schubertdef} is well defined.  It was proved separately by Billey, Jockusch and Stanley in \cite{BJS1993}, and Fomin and Stanley in \cite{FoSt94}, that Schubert polynomials are nonnegative sums of monomials.  For more details on Schubert polynomials and their properties, see~\cite{McD91,Manivel}.

\begin{example} Let $n=4$ and $w=[1,4,3,2]=s_2s_3s_2\in W=S_4$.  We have
$$\mathfrak{S}_{[1,4,3,2]}(\bfx) = \partial_2\partial_3\partial_2(x_1^3x_2^2x_3)=x_2^2x_3+x_1^2x_3+x_1^2x_2+x_1x_2x_3+x_1x_2^2 .$$
Note that $x_2^2x_3$ is the minimal term appearing in the expansion above with respect to our monomial ordering, and $x_2^2x_3 = \mathbf{x}^\gamma$ where $\gamma=(0,2,1,0)$ is the Lehmer code of $w$.
\end{example}

As noted in the example above, the smallest monomial term (with respect to the lexicographical order) appearing in $\mathfrak{S}_w(\mathbf{x})$ is the monomial $\mathbf{x}^\gamma$, where $\gamma$ is a Lehmer code of $w$.  We now prove that this property is true for all Schubert polynomials.

\begin{lem}\label{P:Schubert_upper_triangular}
Let $w\in W$ and $\gamma=(\gamma_1, \gamma_2, \ldots, \gamma_{n-1}, \gamma_n=0)$ denote the Lehmer code of $w$.  Then the Schubert polynomial $\mathfrak{S}_w(\bfx)$ has the expansion:
\begin{eqnarray}\label{eqn.schubert}
\mathfrak{S}_w(\bfx)=\bfx^{\gamma}+\sum_\delta c_{\delta}\bfx^{\delta}
\end{eqnarray}
where $c_{\delta}\neq 0$ implies that $\bfx^{\gamma}<\bfx^{\delta}$.
\end{lem}

\begin{proof} We proceed by induction on the Lehmer code of $w$, which we interpret as the exponent vector of a monomial.  In particular, we induct on degree (i.e.~the number of inversions of $w$) and use the converse of lexicographical order to induct on the Lehmer codes of a given degree.  When $\ell(w)=0$ then $w=e$ so $\mathfrak{S}_e=1$ and the desired expansion of $\mathfrak{S}_w$ holds trivially in this case.

We now assume $\ell(w)>0$ and that there is an expansion of the form~\eqref{eqn.schubert} for every Schubert polynomial $\mathfrak{S}_v(\mathbf{x})$ with $\ell(v)<\ell(w)$ or $\ell(v)= \ell(w)$ and such that the Lehmer code of $v$ is greater than that of $w$.

For any $j<k$, let $t_{j,k}:=s_js_{j+1}\cdots s_{k-2}s_{k-1}s_{k-2}\cdots s_{j+1}s_j$ denote the transposition which swaps $j$ and $k$. Monk's formula for Schubert polynomials implies that for any $k<n$ and $u\in W$, we have
\begin{equation}\label{Eq:Monks_formula}
x_k\mathfrak{S}_{u}(\bfx)=\sum_{\substack{j>k \\ \ell(ut_{k,j})=\ell(u)+1}}\mathfrak{S}_{ut_{k,j}}(\bfx)-\sum_{\substack{j<k\\ \ell(ut_{j,k})=\ell(u)+1}}\mathfrak{S}_{ut_{j,k}}(\bfx).
\end{equation}
Equation~\eqref{Eq:Monks_formula} appears in~\cite[Equation (4.15')]{McD91} and in~\cite[Exercise 2.7.3]{Manivel}.

Let $k_0$ denote the smallest value for which $\gamma_{k_0}\neq 0$ and $(k_0,j_0)\in\Inv(w)$ denote the unique inversion such that $w(k_0)=w(j_0)+1$.  Note that such an inversion exists as $w(i)=i$ for all $i<k_0$ by our assumptions.  In particular, this implies $w(k_0)>k_0$ and $\gamma_{k_0} = w(k_0)-k_0$.
Define $v:=wt_{k_0,j_0}$ and $v':=vt_{k_0-1, k_0}$ (if $k_0=1$, then we disregard $v'$).  In particular, note that $\ell(v)=\ell(w)-1$ and $\ell(v')=\ell(v)+1=\ell(w)$.  Furthermore, our choice of $k_0$ implies that $v'$ is the unique permutation such that $v' = vt_{j,k_0}$ with $j<k_0$ and $\ell(v t_{j,k_0})=\ell(v)+1$. Applying Equation~\eqref{Eq:Monks_formula} with $u=v$ and $k=k_0$ now gives us
\begin{equation}\label{Eq:Monks_formula2}
\mathfrak{S}_{w}(\bfx)=x_{k_0}\mathfrak{S}_{v}(\bfx)+\mathfrak{S}_{v'}(\bfx)-\sum_{w'}\mathfrak{S}_{w'}(\bfx)
\end{equation}
where the sum is taken over all $w'=vt_{k_0,j}$ with $\ell(w')=\ell(w)$, $j>k_0$ and $j\neq j_0$.

Let $\delta$ and $\delta'$ denote codes of $v$ and $v'$, respectively, and let $\gamma'$ denote the code of $w'$ for some $w'$ appearing in the sum.  It is easy to check that
$$\delta_i=\begin{cases}\gamma_i & \text{if $i\neq k_0$}\\ \gamma_i-1 & \text{if $i=k_0$}\end{cases}\quad\text{and}\quad \delta'_i=\begin{cases}\gamma_i & \text{if $i\notin \{k_0-1,k_0\}$}\\ \gamma_{k_0} & \text{if $i=k_0-1$}\\ 0 & \text{if $i=k_0$.}\end{cases}$$
In particular, $\bfx^\gamma=x_{k_0}\bfx^{\delta}$ and $\bfx^\gamma <\bfx^{\delta'}$.

We now have only to show that $\bfx^{\gamma}< \bfx^{\gamma'}$. To start, recall that $w'=vt_{k_0,j}$ with $\ell(w') = \ell(v)+1$ and $k_0<j$.  This implies $v(j)>v(k_0) = w(k_0)-1$. We furthermore know that $j\neq j_0$, so $v(j)\neq v(j_0)=w(k_0)$.  Thus $w'(k_0) = v(j) > w(k_0)$.  By construction,  $\gamma'_i=\gamma_i=0$ for all $i<k_0$ which implies $w'(i)=i$ for all $i<k_0$.  Since $w'(k_0)>w(k_0)>k_0$, we obtain
\[
\gamma_{k_0}' = w'(k_0)-k_0 > w(k_0)-k_0 = \gamma_{k_0}
\]
as desired.
This proves $\bfx^\gamma <\bfx^{\gamma'}$ for any $w'$ with code $\gamma'$ appearing in the sum from~\eqref{Eq:Monks_formula2}.  The lemma now follows by induction.
\end{proof}

\begin{example}
In this example, we illustrate Equation \eqref{Eq:Monks_formula2} from the proof of Lemma \ref{P:Schubert_upper_triangular}.  Let $w=[1,5,3,6,2,4]$.  Then $w$ has Lehmer code $(0,3,1,2,0,0)$.  Using the notation in the proof of Lemma \ref{P:Schubert_upper_triangular}, $(k_0,j_0)=(2,6)$ and $(w(k_0),w(j_0))=(5,4)$.  We also have
$$v=wt_{2,6}=[1,4,3,6,2,5],\quad\text{and}\quad v'=vt_{1,2}=[4,1,3,6,2,5].$$
In this case, the sum in Equation \eqref{Eq:Monks_formula2} contains only one summand with $w'=vt_{2,4}=[1,6,3,4,2,5]$ yielding:
$$\mathfrak{S}_{[1,5,3,6,2,4]}(\bfx)=x_{2}\mathfrak{S}_{[1,4,3,6,2,5]}(\bfx)+\mathfrak{S}_{[4,1,3,6,2,5]}(\bfx)-\mathfrak{S}_{[1,6,3,4,2,5]}(\bfx).$$
The codes of $v$, $v'$ and $w'$ are respectively $(0,2,1,2,0,0)$, $(3,0,1,2,0,0)$ and $(0,4,1,1,0,0)$.
\end{example}

\begin{rem}\label{rem.monomials}
Lemma~\ref{P:Schubert_upper_triangular} is analogous to Billey and Haiman's Lemma 4.11 in~\cite{Billey-Haiman1995} which states that $\bfx^{\gamma}$ is the leading term (i.e.~maximal monomial) in the expansion of $\mathfrak{S}_w(\bfx)$ when imposing \emph{reverse lexicographical} order on the monomials.  It should be noted that reverse lexicographical order is not the converse of lexicographical order, so \cite[Lemma 4.11]{Billey-Haiman1995} does not directly imply Lemma \ref{P:Schubert_upper_triangular}.  However the proof of Lemma \ref{P:Schubert_upper_triangular} given above is modeled after Mcdonald's proof of \cite[Lemma 4.11]{Billey-Haiman1995} which appears in \cite[(4.16)]{McD91}.  The main difference in the proof Lemma \ref{P:Schubert_upper_triangular} above is that we use the ``smallest" inversion (i.e.~in the proof of Lemma~\ref{P:Schubert_upper_triangular} we take $k_0$ to be the smallest value for which $\gamma_{k_0}\neq 0$) and not the ``largest".  Observe that our argument using the ``smallest" inversion, as seen in Equation \eqref{Eq:Monks_formula2}, does not yield a manifestly positive formula for the expansion of Schubert polynomials as a sum of monomials.  However, the induction used to prove \cite[(4.16)]{McD91} does give a positive formula which is stated as a corollary in \cite[(4.19)]{McD91}.
\end{rem}

We can now prove our main theorem, which shows that the images of the Schubert polynomials corresponding to elements from $W(\alpha)$ form a basis of $H^*(\B^\alpha)$.

\begin{proof}[Proof of Theorem \ref{T:Schubert_Basis}]
Let $w\in W(\alpha)$.  Then there exists a unique $\tabone\in\RSCT_n(\alpha)$ for which $w=u_\tabone$.  Let $\gamma$ denote the Lehmer code of $w$, which is also the inversion vector of $\tabone$.  By Lemma~\ref{P:Schubert_upper_triangular}, we can write
$$\mathfrak{S}_w(\bfx)=\bfx^{\gamma}+\sum_\delta c_{\delta}\bfx^{\delta}$$
where the sum is over compositions $\delta$ and $\bfx^{\gamma}<\bfx^{\delta}$ for all $c_\delta\neq 0$.  We now have that
$$\pi_0(\mathfrak{S}_w(\bfx))=\pi_0(\bfx^{\gamma})+\sum_\delta c_{\delta}\,\pi_0(\bfx^{\delta}).$$
Since $\gamma$ is the inversion vector of $\tabone$, we have $\pi_0(\bfx^{\gamma})=\bfx^{\gamma}=\bfx^{\tabone}$.  Furthermore, Proposition \ref{P:monomial_expansion} implies
$$
\pi_0(\bfx^{\delta})=\sum d_\tabtwo\, \bfx^\tabtwo
$$
where the sum is over $\tabtwo\in\RSCT_n(\alpha)$ where $\bfx^\tabtwo\geq \bfx^\delta>\bfx^\tabone$ and hence we can write
$$
\pi_0(\mathfrak{S}_w(\bfx))=\bfx^\tabone+\sum_{\tabone<\tabtwo} g_\tabtwo\, \bfx^\tabtwo
$$
for some coefficients $g_\tabtwo$.  This equation implies that the transition matrix from the set $\{\pi_0(\mathfrak{S}_{w}(\bfx))\mid w\in W(\alpha)\}$ to the basis $\{x^\tabone \mid \tabone \in \RSCT_n(\alpha)\}$ of $\ca_0$ is invertible.  In fact, it is upper triangular with $1$'s on the diagonal.  This proves the theorem.
\end{proof}

Let $\mathfrak{S}_w(\mathbf{y}, \mathbf{x})\in \C[\ft\times \ft]$ denote the \textbf{double Schubert polynomial} indexed by $w\in W$; see \cite{Manivel} for the definition.  As a corollary of Theorem~\ref{T:Schubert_Basis}, we obtain the corresponding statement for equivariant cohomology.

\begin{cor}\label{cor.equiv-main-thm} The set $\{\pi(\mathfrak{S}_w(\mathbf{y}, \mathbf{x})) \mid w\in W(\alpha)\}$ forms an $S(\fs^*)$-module basis of the equivariant cohomology $H_S^*(\B^\alpha)\simeq \ca$.
\end{cor}
\begin{proof} By Theorem~\ref{T:Schubert_Basis}, the polynomials $\{\pi_0(\mathfrak{S}_{w}(\mathbf{x})) \mid w\in W(\alpha)\}$ form a homogeneous basis of $H^*(\B^\alpha)$.  Since $\ev\circ\,\pi(\mathfrak{S}_w(\mathbf{y}, \mathbf{x})) = \pi_0(\mathfrak{S}_w(\mathbf{x}))$, the result now follows from Lemma~\ref{P:equiv_basis}.
\end{proof}

\begin{example}\label{ex.(2,2)-Schuberts}
Let $n=4$ and $\alpha=(2,2)$.  We calculate the image of each Schubert polynomial $\mathfrak{S}_w(\bfx)$ under $\pi_0$. We first recall the set $\RSCT_4(\alpha)$ and corresponding Springer monomial basis of $H^*(\B^\alpha)$; this data is displayed in the table below (c.f.~Example~\ref{ex.(2,2)}).
\[
\begin{tabular}{c|cccccc}
 $\tabone$ &
\begin{ytableau} 3 & 1\\ 4 & 2\\ \end{ytableau}&
\begin{ytableau} 4 & 1\\ 3 & 2\\ \end{ytableau}&
\begin{ytableau} 2 & 1\\ 4 & 3\\ \end{ytableau}&
\begin{ytableau} 3 & 2\\ 4 & 1\\ \end{ytableau}&
\begin{ytableau} 4 & 2\\ 3 & 1\\ \end{ytableau}&
\begin{ytableau} 4 & 3\\ 2 & 1\\ \end{ytableau}\\ \\ \hline \\
$\bfx^{\tabone}$ & $1$ &$x_3$ & $x_2$&$x_1$ & $x_1x_3$ & $x_1x_2$\\
\end{tabular}
\]
By degree considerations, it suffices to calculate $\pi_0(\mathfrak{S}_{w}(\bfx))$ for $\ell(w)\leq 2$ (if $\ell(w)\geq 3$ then $\pi_0(\mathfrak{S}_w(\bfx))=0$). We obtain the following; note that the last column records whether or not $w$ is an element of $W(\alpha)$.
\[
\begin{tabular}{c|c|c|c}
$w$ & $\mathfrak{S}_{w}(\bfx)$ & $\pi_0(\mathfrak{S}_{w}(\bfx))$ & $W(\alpha)$\\ \hline
$e$ & $1$ & $1$ & yes\\
$s_3$ & $\underline{x_3}+x_2+x_1$ & $\underline{x_3}+x_2+x_1$ & yes\\
$s_2$ & $\underline{x_2}+x_1$ & $\underline{x_2}+x_1$& yes \\
$s_2s_3$ & $\underline{x_2x_3}+x_1x_3+x_1x_2$ & $0$ & no\\
$s_3s_2$ & $\underline{x_2^2}+ x_1x_2 +x_1^2$ & $x_1x_2$ & no\\
$s_1$ & $\underline{x_1}$ & $\underline{x_1}$ & yes \\
$s_1s_3$ & $\underline{x_1x_3}+x_1x_2+x_1^2$ & $\underline{x_1x_3}+x_1x_2$& yes\\
$s_1s_2$ & $\underline{x_1x_2}$ & $\underline{x_1x_2}$ & yes \\
$s_2s_1$ & $\underline{x_1^2}$ & $0$ & no\\
\end{tabular}
\]
For each Schubert polynomial $\mathfrak{S}_{w}(\bfx)$, we have underlined the minimal monomial $\bfx^{\gamma}$, so $\gamma$ is the Lehmer code of $w$ as in Lemma~\ref{P:Schubert_upper_triangular}.
\end{example}

We make two observations from Example \ref{ex.(2,2)-Schuberts}.  The first is that the set $W(\alpha)$ does not uniquely satisfy the basis property from Theorem \ref{T:Schubert_Basis}.  In particular, $\pi_0(\mathfrak{S}_{s_1s_2}(\bfx))=\pi_0(\mathfrak{S}_{s_3s_2}(\bfx))$ and hence replacing $s_1s_2$ with $s_3s_2$ in $W(\alpha)$ also corresponds to a basis of $H^*(\B^\alpha)$.  The second is that each polynomial $\pi_0(\mathfrak{S}_{w}(\bfx))$ is a non-negative sum of Springer monomials.  This motivates the following question.

\begin{question}\label{quest.positivity}
Let $\alpha$ be a composition of $n$ and $w\in W$ and write
\begin{eqnarray}\label{eqn.Schubert.constants}
\pi_0(\mathfrak{S}_{w}(\bfx))=\sum_{\tabone\in\RSCT_n(\alpha)} g_{\tabone}\, \bfx^{\tabone}.
\end{eqnarray}
Do we have $g_{\tabone}\in \Z_{\geq 0}$ for all $\tabone\in\RSCT_n(\alpha)$?
\end{question}

Note that negative terms can appear in the expansion formula for the image of an individual monomial $\pi_0(\bfx^\delta)$, as seen in Example \ref{Ex:P_delta_expansion}.  Looking more closely at the calculation of $\pi_0(\mathfrak{S}_{s_2s_3}(\bfx))$ in Example \ref{ex.(2,2)-Schuberts} we find that
$$\pi_0(\mathfrak{S}_{s_2s_3}(\bfx))=\pi_0(x_2x_3)+\pi_0(x_1x_3)+\pi_0(x_1x_2)=(-x_1x_3-x_1x_2)+x_1x_3+x_1x_2=0.$$
Observe that $x_1x_3, x_1x_2\in\GP_{(2,2)}$ while the monomial $x_2x_3$ is not an element of $\GP_{(2,2)}$.  This example shows that although the structure constants are nonnegative in many examples, there is typically some cancellation to take into account.  The answer to Question~\ref{quest.positivity} is known to be `yes' in the special case that $\alpha$ is one of $(n)$, $(n-1,1)$, or $(1,1,\ldots, 1)$.  Note that when $\alpha=(1,1,\ldots, 1)$, the Springer fiber $\B^\alpha$ is the full flag variety, and we obtain a positive answer to Question~\ref{quest.positivity} using the formulas from~\cite{Billey-Haiman1995, McD91, FoSt94} .


\section{Connections with the geometry of Springer fibers}\label{sec.geometry}

It is well known that the Schubert polynomial $\mathfrak{S}_w(\mathbf{x})$ is a polynomial representative for the fundamental cohomology class of the Schubert variety $\B_{w_0w}:=\overline{Bw_0wB/B}$ where $w_0$ denotes the longest element of $W$.  It is therefore natural to ask if the polynomials $\phi_0^*(\mathfrak{S}_w(\mathbf{x}))$ represent a fundamental cohomology class of a subvariety in the Springer fiber $\B^\lambda$.  Unfortunately, due the fact that Springer fibers are usually singular, the classical notion of a fundamental cohomology class of a subvariety using Poincar\'{e} duality is not defined.  However, the notion of a fundamental homology class of a subvariety is well defined (see \cite{DLP1988} or \cite[Appendix B]{Fulton1997}).

We briefly recall the connections between homology classes and Schubert polynomials for the flag variety $\B$, which is smooth.  For any subvariety $Z\subseteq \B$, let $[Z]$ denote the corresponding fundamental homology class in $H_*(\B)$.  Since $\B$ is smooth, the Poincar\'{e} duality isomorphism implies that for each class $[Z]$, there exists a unique cohomology class $\sigma_Z$ for which
$$\sigma_Z\cap [\B]=[Z].$$
Here $\cap [\B]: H^*(\B)\rightarrow H_*(\B)$ denotes the cap product with the top fundamental class $[\B]$.  For the Schubert variety, the class $\sigma_w:=\sigma_{\B_{w_0w}}$ can be represented by the Schubert polynomial $\mathfrak{S}_w(\mathbf{x})$ using the Borel presentation of $H^*(\B)$.

Recall that the inclusion $\phi: \B^\lambda\hookrightarrow \B$ induces a surjective map $\phi_0^*: H^*(\B)\to H^*(\B^\lambda)$.  Since we do not consider equivariant cohomology in this section, we denote will denote $\phi_0^*$ by just $\phi^*$.  We now give a geometric interpretation of the classes $\phi^*(\sigma_w)\in H^*(\B^\lambda)$, which are represented by the polynomials $\phi^*(\mathfrak{S}_w(\bfx)) = \pi_0(\mathfrak{S}_w(\mathbf{x}))$ in $\ca_0$.  Note that the following proposition is true for any subvariety $X$ of the flag variety with inclusion map $\phi:X\hookrightarrow \B$ (not just Springer fibers).

\begin{prop}\label{P:homology_classes}
Let $\sigma_{w}\in H^*(\B)$ denote the fundamental cohomology class of the Schubert variety $\B_{w_0w}$. Then

\begin{equation}\label{Eq:Geometry_sp}
\phi^*(\sigma_w)\cap [\B^\lambda]=[\phi^{-1}(g\B_{w_0w})]\in H_*(\B^\lambda)
\end{equation}
for generic $g\in G$.
\end{prop}

\begin{proof}
Let $r:\tilde\B\rightarrow \B^\lambda$ denote a desingularization of the Springer fiber $\B^\lambda$.  Note that the fact that such a resolution exists is due to Hironaka \cite{Hironaka1964}.  Let $f:=\phi\circ r:\tilde\B\rightarrow \B$ and consider the diagram

$$
\begin{tikzcd}
 & H^*(\tilde\B) \arrow[r, "\cap{[\tilde\B]}"]  & H_*(\tilde\B) \arrow[d, "r_*"] \\
H^*(\B) \arrow[ur, "f^*"] \arrow[r, "\phi^*"]  & H^*(\B^\lambda)\arrow[u, "r^*"] \arrow[r, "\cap{[\B^\lambda]}"] & H_*(\B^\lambda)
\end{tikzcd}
$$
Since $r:\tilde\B\rightarrow \B^\lambda$ is a surjective, birational morphism, we must have that $r_*([\tilde\B])=[\B^\lambda]$.  This implies
$$r_*(f^*(\sigma_w)\cap [\tilde\B])=r_*(r^*\circ \phi^*(\sigma_w)\cap [\tilde\B])=\phi^*(\sigma_w)\cap r_*([\tilde\B])=\phi^*(\sigma_w)\cap [\B^\lambda]$$
which is the left hand side of Equation \eqref{Eq:Geometry_sp}.

Kleiman's transversality theorem \cite{Kleiman1974} implies that for generic $g\in G$, the preimage $f^{-1}(g\B_{w_0w})$ is generically transverse. By \cite[Theorem 1.23]{EH2016}, we have that $f^*([g\B_{w_0w}])=[f^{-1}(g\B_{w_0w})]$ as elements in the Chow ring of $\tilde\B$ (graded by codimension).  Since smooth pullback commutes with the cycle map from the Chow ring to cohomology \cite[Corollary 19.2]{Fulton1984}, it follows that $f^*(\sigma_w)=\sigma_{f^{-1}(g\B_{w_0w})}$ in $H^*(\tilde\B).$  (Note that the maps $r,\phi$ are proper morphisms and $f:\tilde\B\rightarrow \B$ is a proper morphism between smooth varieties.  Hence proper pushforward and smooth pullback are well defined on Chow groups/rings).    We now have that
$$r_*(f^*(\sigma_w)\cap [\tilde\B])=r_*(\sigma_{f^{-1}(g\B_{w_0w})}\cap [\tilde\B])=r_*([f^{-1}(g\B_{w_0w})]).$$
Again, by Kleiman's transversality theorem, the varieties $f^{-1}(g\B_{w_0w})$ and $\phi^{-1}(g\B_{w_0w})$ are both generically reduced and of the same codimension.  Since $r:\tilde\B\rightarrow\B^\lambda$ is birational, the varieties $f^{-1}(g\B_{w_0w})$ and $\phi^{-1}(g\B_{w_0w})$ are also of the same dimension and hence
$$r_*([f^{-1}(g\B_{w_0w})])=[\phi^{-1}(g\B_{w_0w})]$$
which completes the proof.
\end{proof}

Observe that the variety $\phi^{-1}(g\B_{w})$ is simply the intersection $\B^{\lambda}\cap g\B_{w}\subseteq\B^{\lambda}$.  If we let $C_w:=BwB/B$ denote the open Schubert cell, then it is known that for carefully chosen $g'\in G$, the collection of nonempty intersections $\{\B^{\lambda}\cap g'C_w\mid w\in W \}$ is an affine paving of $\B^{\lambda}$ \cite{Shimomura1980, Tymoczko2006}.  In these cases, the corresponding nonzero homology classes $\{[\B^{\lambda}\cap g'\B_w]\mid w\in W\}$ form a basis of $H_*(\B^{\lambda})$.  We remark that the generic condition of $g\in G$ in Proposition~\ref{P:homology_classes} typically excludes any $g'$ such that $\{\B^{\lambda}\cap g'C_w\mid w\in W \}$ is an affine paving of $\B^{\lambda}$.  Indeed, otherwise the map $\cap [\B^\lambda]: H^*(\B^\lambda)\rightarrow H_*(\B^\lambda)$ would be an isomorphism and imply that the Poincar\'e polynomials of Springer fibers are palindromic, which is false in most cases.

One immediate consequence of Proposition~\ref{P:homology_classes} is that linear relations among the classes $\{\phi^*(\sigma_w)\mid w\in W\}$ in $H^*(\B^{\lambda})$ translate to linear relations on $\{[\B^{\lambda}\cap g\B_{w_0w}] \mid w\in W\}$ in $H_*(\B^{\lambda})$.

\begin{cor}\label{cor.homology_linear_combs}
Let $g\in G$ be generic and suppose that $\sum_{w\in W} c_w\, \phi^*(\sigma_w)=0$ in $H^*(\B^{\lambda})$ for some coefficients $c_w\in \C$.  Then $$\sum_{w\in W} c_w\, [\B^{\lambda}\cap g\B_{w_0w}]=0$$ in $H_*(\B^{\lambda})$.  In particular, if $\phi^*(\sigma_w)=0$, then $[\B^{\lambda}\cap g\B_{w_0w}]=0$.
\end{cor}

The converse of this statement is not true since $\cap [\B^\lambda]: H^*(\B^\lambda)\rightarrow H_*(\B^\lambda)$ is usually not an isomorphism.

Note that $\phi^*(\sigma_w)$ can be computed explicitly by expanding its polynomial representative $\pi_0(\mathfrak{S}_w(\bfx))$ in terms of the Springer monomial basis using Theorem~\ref{T:P_basis_det_formula}.  Hence Corollary~\ref{cor.homology_linear_combs} gives a combinatorially sufficient condition to determine if the homology classes $[\B^{\lambda}\cap g\B_{w_0w}]=0$ for generic $g\in G$.


\begin{thebibliography}{10}

\bibitem{AH2016}
Hiraku Abe and Tatsuya Horiguchi.
\newblock The torus equivariant cohomology rings of {S}pringer varieties.
\newblock {\em Topology Appl.}, 208:143--159, 2016.

\bibitem{Billey-Haiman1995}
Sara Billey and Mark Haiman.
\newblock {S}chubert polynomials for the classical groups.
\newblock {\em J. Amer. Math. Soc.}, 8(2):443--482, 1995.

\bibitem{Bi96}
Sara~C. Billey.
\newblock Kostant polynomials and the cohomology ring for {$G/B$}.
\newblock {\em Duke Math. J.}, 96(1):205--224, 1999.

\bibitem{BJS1993}
Sara~C. Billey, William Jockusch, and Richard~P. Stanley.
\newblock Some combinatorial properties of {S}chubert polynomials.
\newblock {\em Journal of Algebraic Combinatorics}, 2(4):345--374, Nov 1993.

\bibitem{Bjorner-Brenti}
Anders Bj\"{o}rner and Francesco Brenti.
\newblock {\em Combinatorics of {C}oxeter groups}, volume 231 of {\em Graduate
  Texts in Mathematics}.
\newblock Springer, New York, 2005.

\bibitem{Ca86}
James~B. Carrell.
\newblock Orbits of the {W}eyl group and a theorem of {D}e{C}oncini and
  {P}rocesi.
\newblock {\em Compositio Math.}, 60(1):45--52, 1986.

\bibitem{DLP1988}
C.~De~Concini, G.~Lusztig, and C.~Procesi.
\newblock Homology of the zero-set of a nilpotent vector field on a flag
  manifold.
\newblock {\em J. Amer. Math. Soc.}, 1(1):15--34, 1988.

\bibitem{DP81}
Corrado De~Concini and Claudio Procesi.
\newblock Symmetric functions, conjugacy classes and the flag variety.
\newblock {\em Invent. Math.}, 64(2):203--219, 1981.

\bibitem{DH}
Barry Dewitt and Megumi Harada.
\newblock Poset pinball, highest forms, and {$(n-2,2)$} {S}pringer varieties.
\newblock {\em Electron. J. Combin.}, 19(1):Paper 56, 35, 2012.

\bibitem{Drellich2015}
Elizabeth Drellich.
\newblock Monk's rule and {G}iambelli's formula for {P}eterson varieties of all
  {L}ie types.
\newblock {\em J. Algebraic Combin.}, 41(2):539--575, 2015.

\bibitem{DF}
David~S. Dummit and Richard~M. Foote.
\newblock {\em Abstract algebra}.
\newblock John Wiley \& Sons, Inc., Hoboken, NJ, third edition, 2004.

\bibitem{EH2016}
David Eisenbud and Joe Harris.
\newblock {\em 3264 and all that---a second course in algebraic geometry}.
\newblock Cambridge University Press, Cambridge, 2016.

\bibitem{FoSt94}
Sergey Fomin and Richard~P. Stanley.
\newblock Schubert polynomials and the nil-{C}oxeter algebra.
\newblock {\em Adv. Math.}, 103(2):196--207, 1994.

\bibitem{Fresse2009}
Lucas Fresse.
\newblock Betti numbers of {S}pringer fibers in type {$A$}.
\newblock {\em J. Algebra}, 322(7):2566--2579, 2009.

\bibitem{Fresse2009-2}
Lucas Fresse.
\newblock Singular components of {S}pringer fibers in the two-column case.
\newblock {\em Ann. Inst. Fourier (Grenoble)}, 59(6):2429--2444, 2009.

\bibitem{Fulton1984}
William Fulton.
\newblock {\em Intersection theory}, volume~2 of {\em Ergebnisse der Mathematik
  und ihrer Grenzgebiete (3) [Results in Mathematics and Related Areas (3)]}.
\newblock Springer-Verlag, Berlin, 1984.

\bibitem{Fulton1997}
William Fulton.
\newblock {\em Young tableaux}, volume~35 of {\em London Mathematical Society
  Student Texts}.
\newblock Cambridge University Press, Cambridge, 1997.
\newblock With applications to representation theory and geometry.

\bibitem{Fung2003}
Francis~Y.C. Fung.
\newblock On the topology of components of some {S}pringer fibers and their
  relation to {K}azhdan--{L}usztig theory.
\newblock {\em Advances in Mathematics}, 178(2):244 -- 276, 2003.

\bibitem{GP92}
A.~M. Garsia and C.~Procesi.
\newblock On certain graded {$S_n$}-modules and the {$q$}-{K}ostka polynomials.
\newblock {\em Adv. Math.}, 94(1):82--138, 1992.

\bibitem{Graham-Zierau2011}
William Graham and R.~Zierau.
\newblock Smooth components of {S}pringer fibers.
\newblock {\em Ann. Inst. Fourier (Grenoble)}, 61(5):2139--2182 (2012), 2011.

\bibitem{HT11}
Megumi Harada and Julianna Tymoczko.
\newblock A positive {M}onk formula in the {$S^1$}-equivariant cohomology of
  type {$A$} {P}eterson varieties.
\newblock {\em Proc. Lond. Math. Soc. (3)}, 103(1):40--72, 2011.

\bibitem{HT17}
Megumi Harada and Julianna Tymoczko.
\newblock Poset pinball, {GKM}-compatible subspaces, and {H}essenberg
  varieties.
\newblock {\em J. Math. Soc. Japan}, 69(3):945--994, 2017.

\bibitem{Hironaka1964}
Heisuke Hironaka.
\newblock Resolution of singularities of an algebraic variety over a field of
  characteristic zero. {I}, {II}.
\newblock {\em Ann. of Math. (2) {\bf 79} (1964), 109--203; ibid. (2)},
  79:205--326, 1964.

\bibitem{Kleiman1974}
Steven~L. Kleiman.
\newblock The transversality of a general translate.
\newblock {\em Compositio Math.}, 28:287--297, 1974.

\bibitem{Kraft1981}
Hanspeter Kraft.
\newblock Conjugacy classes and {W}eyl group representations.
\newblock In {\em Young tableaux and {S}chur functors in algebra and geometry
  ({T}oru\'{n}, 1980)}, volume~87 of {\em Ast\'{e}risque}, pages 191--205. Soc.
  Math. France, Paris, 1981.

\bibitem{KP12}
Shrawan Kumar and Claudio Procesi.
\newblock An algebro-geometric realization of equivariant cohomology of some
  {S}pringer fibers.
\newblock {\em J. Algebra}, 368:70--74, 2012.

\bibitem{LaSc81}
Alain Lascoux and Marcel-Paul Sch\"{u}tzenberger.
\newblock Polyn\^{o}mes de {S}chubert.
\newblock {\em C. R. Acad. Sci. Paris S\'{e}r. I Math.}, 294(13):447--450,
  1982.

\bibitem{McD91}
I.G. Macdonald.
\newblock Notes on {S}chubert polynomials.
\newblock {\em Publications du L.A.C.I.M.}, 6, 1991.

\bibitem{Manivel}
Laurent Manivel.
\newblock {\em Symmetric functions, {S}chubert polynomials and degeneracy
  loci}, volume~6 of {\em SMF/AMS Texts and Monographs}.
\newblock American Mathematical Society, Providence, RI; Soci\'{e}t\'{e}
  Math\'{e}matique de France, Paris, 2001.

\bibitem{Mb10}
Aba Mbirika.
\newblock A {H}essenberg generalization of the {G}arsia-{P}rocesi basis for the
  cohomology ring of {S}pringer varieties.
\newblock {\em Electron. J. Combin.}, 17(1):Research Paper 153, 29, 2010.

\bibitem{MT13}
Aba Mbirika and Julianna Tymoczko.
\newblock Generalizing {T}anisaki's ideal via ideals of truncated symmetric
  functions.
\newblock {\em J. Algebraic Combin.}, 37(1):167--199, 2013.

\bibitem{PT19}
Martha Precup and Julianna Tymoczko.
\newblock {S}pringer fibers and {S}chubert points.
\newblock {\em European Journal of Combinatorics}, 76:10 -- 26, 2019.

\bibitem{Shimomura1980}
Naohisa Shimomura.
\newblock A theorem on the fixed point set of a unipotent transformation on the
  flag manifold.
\newblock {\em J. Math. Soc. Japan}, 32(1):55--64, 1980.

\bibitem{Spaltenstein1976}
N.~Spaltenstein.
\newblock The fixed point set of a unipotent transformation on the flag
  manifold.
\newblock {\em Nederl. Akad. Wetensch. Proc. Ser. A {\bf 79}=Indag. Math.},
  38(5):452--456, 1976.

\bibitem{Springer1976}
T.~A. Springer.
\newblock Trigonometric sums, {G}reen functions of finite groups and
  representations of {W}eyl groups.
\newblock {\em Invent. Math.}, 36:173--207, 1976.

\bibitem{Springer1978}
T.~A. Springer.
\newblock A construction of representations of {W}eyl groups.
\newblock {\em Invent. Math.}, 44(3):279--293, 1978.

\bibitem{Tanisaki1982}
Toshiyuki Tanisaki.
\newblock Defining ideals of the closures of the conjugacy classes and
  representations of the {W}eyl groups.
\newblock {\em T\^ohoku Math. J. (2)}, 34(4):575--585, 1982.

\bibitem{Tymoczko2006}
Julianna Tymoczko.
\newblock Linear conditions imposed on flag varieties.
\newblock {\em Amer. J. Math.}, 128(6):1587--1604, 2006.

\end{thebibliography}
\end{document}